\documentclass{amsart}

\usepackage{amscd}
\usepackage{amssymb, amsmath}

\usepackage{tikz}
\usetikzlibrary{shapes.multipart, positioning,
decorations.pathreplacing, decorations.markings,
decorations.pathmorphing, calc, matrix, cd}

\usepackage[all, knot]{xy}
\usepackage{hyperref}
 \usepackage[calc]{picture}
\usepackage{graphicx}
\usepackage{scalerel}

\theoremstyle{plain} 
\newtheorem{thm}{Theorem}[section]

\newtheorem*{introthm*}{Theorem}

\newtheorem{cor}[thm]{Corollary}
\newtheorem{lem}[thm]{Lemma}

\newtheorem{prop}[thm]{Proposition}

\theoremstyle{definition}
\newtheorem{defn}[thm]{Definition}

\newlength{\thmtopspace}                
\newlength{\thmbotspace}                
\newlength{\thmheadspace}               
\newlength{\thmindent}                  
\theoremstyle{remark}

\newtheorem*{rems}{Remark}

\makeatletter
\newsavebox\myboxA
\newsavebox\myboxB
\newlength\mylenA
\newcommand*\xoverline[2][0.75]{%
    \sbox{\myboxA}{$\m@th#2$}%
    \setbox\myboxB\null
    \ht\myboxB=\ht\myboxA%
    \dp\myboxB=\dp\myboxA%
    \wd\myboxB=#1\wd\myboxA
    \sbox\myboxB{$\m@th\overline{\copy\myboxB}$}
    \setlength\mylenA{\the\wd\myboxA}
    \addtolength\mylenA{-\the\wd\myboxB}%
    \ifdim\wd\myboxB<\wd\myboxA%
       \rlap{\hskip 0.5\mylenA\usebox\myboxB}{\usebox\myboxA}%
    \else
        \hskip -0.5\mylenA\rlap{\usebox\myboxA}{\hskip 0.5\mylenA\usebox\myboxB}%
    \fi}

\newcommand*\xunderline[2][0.75]{%
    \sbox{\myboxA}{$\m@th#2$}%
    \setbox\myboxB\null
    \ht\myboxB=\ht\myboxA%
    \dp\myboxB=\dp\myboxA%
    \wd\myboxB=#1\wd\myboxA
    \sbox\myboxB{$\m@th\underline{\copy\myboxB}$}
    \setlength\mylenA{\the\wd\myboxA}
    \addtolength\mylenA{-\the\wd\myboxB}%
    \ifdim\wd\myboxB<\wd\myboxA%
       \rlap{\hskip 0.5\mylenA\usebox\myboxB}{\usebox\myboxA}%
    \else
        \hskip -0.5\mylenA\rlap{\usebox\myboxA}{\hskip 0.5\mylenA\usebox\myboxB}%
    \fi}
\makeatother

\newcommand{\overbar}[1]{\mkern 1.5mu\overline{\mkern-1.5mu#1\mkern-1.5mu}\mkern 1.5mu}
\renewcommand{\o}[1]{\xoverline{#1}}

\makeatletter
\newcommand*\bigcdot{\mathpalette\bigcdot@{.5}}
\newcommand*\bigcdot@[2]{\mathbin{\vcenter{\hbox{\scalebox{#2}{$\m@th#1\bullet$}}}}}
\makeatother

\newcommand{\bN}{\mathbb{N}}

\newcommand{\bZ}{\mathbb{Z}}

\renewcommand{\t}[1]{\widetilde{#1}}

\newcommand{\xra}[1]{\xrightarrow{#1}}

\newcommand{\Hom}[3]{\operatorname{Hom}_{#1}(#2,#3)}

\newcommand{\oHom}{\operatorname{Hom}}

\newcommand{\dhom}{\delta_{\operatorname{Hom}}}
\newcommand{\dab}[1][AB]{d_{#1}}
\newcommand{\doab}[1][\overbar{A}B]{d_{#1}}

\newcommand{\Ai}{$\text{A}_{\infty}$}

\newcommand{\A}[1]{$\text{A}_{#1}$}

\newcommand{\loccit}{\emph{loc.\ \!cit.\ }}

\newcommand{\hhc}[2]{CC^{\bullet} (#1, #2)}
\newcommand{\hhcu}[2]{CC_{\mathsf{u}}^{\bullet} (#1, #2)}

\newcommand{\hhcn}[3][n]{CC^{#1} (#2, #3)}
\newcommand{\hhcln}[3][n]{CC^{\scaleobj{.75}{\leq} #1} (#2, #3)}
\newcommand{\hhclnu}[3][n]{CC^{\scaleobj{.75}{\leq} #1}_{\mathsf{u}} (#2, #3)}

\newcommand{\Tcon}[2][n]{T_{\mathsf{co}}^{#1}(#2)}

\newcommand{\Coder}[1][]{\operatorname{Coder}_{#1}}
\newcommand{\Coderab}[1][\alpha,\beta]{\operatorname{Coder}^{#1}}

\newcommand{\End}{\operatorname{End}}

\newcommand{\Coalgk}[1][k]{\operatorname{Coalg}_{#1}}

\newcommand{\msu}{\mu_{\operatorname{su}}}

\newcommand{\gsu}{g_{\operatorname{su}}}

\newcommand{\nuend}{\nu_{\operatorname{End}}}

\newcommand{\sleq}[1][.75]{\scaleobj{#1}{\leq}}

\newcommand{\muln}[1][n]{\mu^{\sleq #1}}
\newcommand{\nuln}[1][n]{\nu^{\sleq #1}}
\newcommand{\omln}[1][n]{\overbar{\mu}^{\sleq #1}}
\newcommand{\omuln}[1][n]{\overbar{\mu}^{\sleq #1}}

\newcommand{\hln}[1][n]{h^{\sleq #1}}
\newcommand{\rln}[1][n]{r^{\sleq #1}}
\newcommand{\aln}[1][n]{\alpha^{\sleq #1}}
\newcommand{\bln}[1][n]{\beta^{\sleq #1}}
\newcommand{\gamln}[1][n]{\gamma^{\sleq #1}}

\newcommand{\piln}[2][n]{\left(#2\right)^{\sleq #1}}

\usepackage{twoopt}
\newcommandtwoopt{\hsp}[2][\alpha][\beta]{\,{}_{#1}\!\!*_{#2}}

\newcommand{\obs}[1]{o(#1)}

\usepackage{titletoc}

\numberwithin{equation}{section}

\begin{document}
\title{Transfer of \Ai-structures to projective resolutions}
\author{Jesse Burke}
\address{Mathematical Sciences Institute\\ Australian National University\\ Canberra ACT}

\begin{abstract}
We show that an \Ai-algebra structure can be transferred to a
projective resolution of the complex underlying any \Ai-algebra. Under
certain connectedness assumptions, this transferred
structure is unique up to homotopy. In
contrast to the classical results on transfer of \Ai-structures along
homotopy equivalences, our result is of interest when the ground ring
is not a field. We prove an analog for \Ai-module
structures, and both transfer results preserve strict units.
\end{abstract}

\email{jesse.burke@anu.edu.au}
\maketitle

\setcounter{section}{-1}
\section{Introduction}
It is a classical and motivating result in the theory of \Ai-algebras
that \Ai-structures can be transferred along homotopy
equivalences, i.e., if $f: A \xra{\thicksim} B$ is a homotopy equivalence of complexes, and $B$
has an \Ai-algebra structure, then $A$ has an \Ai-algebra structure,
and $f$ can be extended to a morphism of the
\Ai-algebras. Different versions of this result are proved
in \cite{MR580645,MR662761,MR885535, MR1109665, MR1103672, MR1882331,MR1672242,
  LH}. It is most often used the in case $B$ is a
dg-algebra, $A = H_{*}(B)$ is the
homology algebra of $B$ (with zero differential), and the ground
ring $k$ is a field, so there is a homotopy equivalence between $A$ and $B$.

In this paper we show that if $B$ has an \Ai-algebra structure, and
$q: A \to B$ is a projective resolution  of
the complex underlying $B$, over the ground ring $k$, then $A$ has
an \Ai-algebra structure such that $q$ is a strict morphism of
\Ai-algebras (by projective resolution we mean a cofibrant
replacement in the projective model category structure on chain
complexes over $k$). If $B$ satisfies $H_{i}(B) = 0$ for all $i < 0,$ then
the transferred structure on $A$ is unique up to homotopy. If the \Ai-algebra structure on $B$ is strictly unital,
then the transferred structure is also strictly unital, under a mild
assumption on $A$. We prove analogous results for
\Ai-modules.  Note that if $k$ is not a field, then a projective
resolution is generally not a homotopy equivalence, so the earlier
results do not apply.

These transfer results are a technical tool in developing Koszul
duality relative to an arbitrary commutative base ring. By Koszul
duality we mean in the generalized sense of using the bar construction,
or some small replacement of it, to study the homological algebra
of an algebra $B$ (associative, dg, or \Ai), e.g., to construct canonical $B$-projective resolutions of
$B$-modules. If $B$ is not projective as a module over $k,$ then its
bar construction can be nonsensical. For instance, if
$B = k/I$ is a cyclic $k$-algebra, the bar construction of $B$ is an
infinite sequence of copies of $B$ with zero differential. This tells
us nothing about the structure of $B$ and seems to dash any hope of
using the bar construction to construct resolutions. The transfer
results proved here offer an alternative: resolve $B$ over $k$,
transfer the algebra structure to the $k$-resolution, use the bar
construction to construct resolutions there, and then try to massage the
result back down to $B$. This is carried out in detail in the case $B
= k/I$ is cyclic in \cite{1508.03782}, and we hope to return to the
general case in future work.

The proofs of our results use obstruction
theory and the lifting
properties of projective resolutions. Obstruction theory has been used
in the construction of \Ai-structures at least since
\cite{MR662761}. Sullivan, in the context of topology, described
obstruction theory as
\begin{quote}
\ldots much like being in a labyrinth with a weak miner's light
attached to your forehead and being forced always to move forward. The
light enables you to see if you may take your next step but it is not
strong enough to tell you which fork to take when you must make a
decision \cite[\S 6.3]{MR3136262}.
\end{quote}
We can use this analogy to illustrate the relation of the present work to
the classical results on transferring via homotopy equivalence. Our results show that there is always a path
through the labyrinth of
transferring an \Ai-algebra structure to a projective resolution, and
in many cases this path is unique up to homotopy, but we do not
have global information, e.g., a map of the labyrinth. The classical homotopy
transfer results give much more detailed
information, formulated e.g., with SDR data, on the labyrinth of transferring an \Ai-algeba structure
via a homotopy equivalence. Additional
information in specialized situations seems to be needed to better understand
the \Ai-structures that result in the case of resolutions.

Finally, we mention that the results here are easily dualized to give
transfer results for injective resolutions, in the case of
  augmented or nonunital \Ai-algebras. The way we handle the transfer
  of strictly unital (but potentially non-augmented) \Ai-algebras
  assumes there is a free summand in degree zero. It is not clear at
  the moment how to adapt this to the injective case, since injective
  modules do not have free summands. 

\renewcommand{\contentsname}{}
\tableofcontents

\section{Notation and sign conventions}\label{sect:notation}
\begin{enumerate}
\item Throughout, $k$ is a fixed commutative ring. By module, complex, map,
  etc.\ we mean $k$-module, $k$-linear map, etc.  We place no boundedness or
  connectedness assumptions on complexes.

  \item For graded modules
  $M,N,$ $\Hom {} M N$ and $M \otimes N$ are graded by:
  \begin{displaymath}
    \Hom {} M N_n = \prod_{i \in \bZ} \Hom {} {M_i} {N_{i + n}} \quad (M
    \otimes N)_n = {\bigoplus_{i \in \bZ} M_i \otimes N_{n - i}}.
  \end{displaymath}
  If $(M, \delta_{M})$ and $(N, \delta_{N})$ are complexes, then
  $\Hom {} M N$ and $M \otimes N$ are complexes with differentials
  $\dhom$ and $\delta_{\otimes}$ given by:
  \begin{displaymath}
    \dhom(f) = \delta_{N}f - (-1)^{|f|}f \delta_{M} \quad \delta_{\otimes} = \delta_M \otimes 1 + 1 \otimes \delta_N.
  \end{displaymath}
  A morphism of complexes is a cycle in the complex
  $(\Hom {} M N, \dhom).$ A quasi-isomorphism is a morphism that
  induces an isomorphism in homology.

\item All elements of graded objects are assumed to be homogeneous. We
  write $|x|$ for the degree of an element $x$. Set $\Pi$ to be the
  endofunctor of the category of graded modules defined by $(\Pi M)_n =
  M_{n-1}$ and $(\Pi f)[m] = [f(m)]$ for a
  morphism $f$, where $[m] \in (\Pi M)_{n}$ is the element
  corresponding to $m \in M_{n-1}$. There is a degree one natural
  transformation $1 \xra{s} \Pi$ that is the identity on every graded
  module, i.e.,  $s(x) =[x] \in \Pi M.$ If
  $(M, \delta_{M})$ is a complex, we set
  $\delta_{\Pi M} = -s\delta_M s^{-1}$, so
  $s: (M, \delta_{M}) \to (\Pi M, \delta_{\Pi M})$ is a degree one
  cycle in
  complex $(\Hom {} M {\Pi M}, \dhom)$. We write
  $[x_1|\ldots|x_n] = sx_{1}\otimes \ldots\otimes s x_n.$ 

\item Signs are introduced when applying a tensor product of morphisms
  as follows:
  $(f \otimes g)(x \otimes y) = (-1)^{|g||x|}f(x) \otimes g(y)$.
\end{enumerate}

\section{Definitions}
In this section we collect various definitions we need to
 precisely state our main results. See e.g.,
\cite{MR2844537,MR1854636,MR2596638} for an introduction to, and further context
on, \Ai-objects, and \cite{SUal} for an expanded version of the
development below, using string diagrams. Throughout this section $A$
and $B$ denote graded modules.

\begin{defn}\label{defn-of-three-prods-on-hhc}
For $l \in \bN,$ set $\hhcn [l] A B = \Hom {} {(\Pi A)^{\otimes l}} {\Pi B},$
  and for $n \in \bN \cup \{ \infty \},$ set
  $\hhcln A B = \prod_{1 \leq l \leq n} \hhcn [l] A B.$
We write $\hhc A B$ for $\hhcln [\infty] A B$. Elements $f \in \hhc A
B$ are denoted $f = (f^{l})$, with $f^l \in
 \hhcn [l] A B$ \emph{the $l$th tensor homogeneous component}. Since $A$ and $B$ are graded,
so is $\hhcln A B$, using the convention on gradings of Hom and tensor products. We denote the
$i$th homogeneous component of this grading as $\hhcln A B_{i}$.
  \begin{enumerate}
 \item The \emph{Gerstenhaber product} of $\mu = (\mu^{l}) \in \hhc A B_{i}$ and $\nu =
(\nu^{l}) \in \hhc A A_{j}$ is $\mu \circ \nu = ((\mu \circ \nu)^{l}) \in \hhc A B_{i+j},$
with
$$(\mu \circ \nu)^{l} = \sum_{\substack{1 \leq i \leq l\\0 \leq j
    \leq i-1}}  \mu^{i}
(1^{\otimes j} \otimes \nu^{l-i+1} \otimes 1^{\otimes i - j - 1}).$$

\item The \emph{*-product} of $\nu =
(\nu^{l}) \in \hhc B B_{-1}$ and $\alpha = (\alpha^{l}) \in \hhc A B_{0}$ is $\nu*\alpha =
((\nu*\alpha)^{l}) \in \hhc A B_{-1}$, with
\begin{displaymath}
(\nu*\alpha)^{l} = \sum_{\substack{1 \leq j \leq l\\i_{1} + \ldots + i_{j} = l}} \nu^{j}(\alpha^{i_1} \otimes \ldots
\otimes \alpha^{i_{j}}).
\end{displaymath}

\item The \emph{homotopy *-product} of $\nu \in \hhc B B_{-1}$ and $r
  \in \hhc A B_{1}$ with respect to $\alpha, \beta \in \hhc A B_{0}$
  is $(\nu \hsp r) = ((\nu \hsp r)^{l}) \in \hhc {A} B_{0},$ with
\begin{displaymath}
(\nu \hsp r)^{l} = \sum_{\substack{1 \leq i \leq l, 1 \leq k \leq i\\j_{1} +
  \ldots + j_{i} = l}}
 \nu^{i}(\alpha^{j_1}\otimes\ldots\otimes\alpha^{j_{k-1}}\otimes r^{j_k}\otimes
\beta^{j_{k+1}} \otimes \ldots \otimes \beta^{j_i}).
\end{displaymath}
\end{enumerate}
\end{defn}

For any $n
 \geq 1$, $\hhcln A A$ is a submodule and quotient module of $\hhc A
 A.$ The inclusion allows us to apply the three products defined above to elements of
 $\hhcln A A,$ but note that $\hhcln A A$ is not closed under any of
 these operations. We write $\piln {-}: \hhc A A \to \hhcln A A$ for
the canonical projection.

\begin{defn} Let $A$ and $B$ be graded modules and fix $n \in \bN \cup \{\infty\}$.
  \begin{enumerate}
\item  A \emph{nonunital \A n-algebra structure on $A$} is an
element $\nu \in \hhcln A A_{-1}$ such that $\left (\nu \circ
\nu\right )^{\sleq n} = 0.$ (The pair $(A, \nu)$ is then an \A n algebra.)

\item  A \emph{morphism between nonunital \A n algebras
$(A, \nu_{A}) \to (B, \nu_{B})$} is
an element $\alpha \in \hhcln A B_{0}$ such that
$\piln {\nu_{B} * \alpha - \alpha \circ \nu_{A}} = 0.$

\item A \emph{homotopy between morphisms of nonunital \A n-algebras $\alpha, \beta:
  (A, \nu_{A}) \to (B, \nu_{B})$} is an element $r \in \hhcn [n] A
  B_{1}$ such that $\alpha - \beta = \piln{\nu_{B}\hsp r - r \circ \nu_{A}}.$
\end{enumerate}
\end{defn}

The unital versions of the above are defined next. The adjective
strict is used to distinguish from the weaker notion of a homotopy unit,
see e.g., \cite[\S 3.2]{LH}.
\begin{defn}\label{defn:strict-unit}
  Let $A, B$ be graded modules with fixed elements $1 \in A_{0},$ 
  $1 \in B_{0}$ and fix $n \in \bN \cup \{ \infty \}.$
  \begin{enumerate}
  \item An element $\nuln \in \hhcln A A_{-1}$ is
    \emph{strictly unital} (with respect to $1 \in A_{0}$) if
    $\nu^{2}[1|a] = [a] = (-1)^{|a|}[a|1]$ for all $a \in A$ and
    $\nu^{k}[a_{1}|\ldots|a_{i}|1|a_{i+1}|\ldots|a_{k-1}] = 0$ for all
    $a_{1}, \ldots, a_{k-1} \in A$ with $k \neq 2$ and $k \leq n$. An
    \A n-algebra $(A, \nuln)$ is strictly unital if $\nuln$ is
    strictly unital in the above sense.
    
  \item  An
element $\aln \in \hhcln A B_{0}$ is \emph{strictly unital} if
$\alpha^{1}[1] = 1$ and for all $2 \leq k \leq n$ and
$a_1, \ldots, a_{k-1} \in A,$ we have 
$\alpha^{k}[a_1|\ldots|a_{i}|1|a_{i+1}|\ldots|a_{k-1}] = 0.$ A morphism of
strictly unital \A n-algebras is
strictly unital if it is strictly unital in the above sense.

\item An element $\rln \in \hhcln A B_{1}$ is \emph{strictly unital}
  if for all  $1 \leq k \leq n$ and all
$a_1, \ldots, a_{k-1} \in A,$ we have
$r^{k}[a_1|\ldots|a_{i}|1|a_{i+1}|\ldots|a_{k-1}] = 0.$  A homotopy
between strictly unital morphisms is strictly unital if it is strictly unital in
the above sense.
\end{enumerate}
\end{defn}

We now recall definitions related to \A n-modules. It is a small but important
point that \A n-module structures are naturally defined over \A {n-1}-algebras
(as opposed to \A n-algebras).

\begin{defn}
  Let $A$ be a graded module and $(M, \delta_{M})$ a complex.
  \begin{enumerate}
  \item The graded module $\Hom {} M M$ is a canonically a dg algebra, with multiplication given by
    composition and differential $\dhom = [\delta_{M}, -].$
    We denote by $(\End M, \nuend)$ the corresponding \Ai-algebra, so
    $\nuend^{1} = \delta_{\Pi \oHom} = -s\dhom s^{-1}$ and
    $\nuend^{2} = -s\gamma(s^{-1})^{\otimes 2},$ where $\gamma$
    denotes composition. See also \cite[Remark
    2.10]{SUal}.
  
    \item If $(A, \nu^{\leq n-1})$ is a nonunital \A {n-1}-algebra, an \emph{\A n-module structure on
      $(M, \delta_{M})$} is a morphism of \A {n-1}-algebras
    $p_{M} \in \hhcln [n-1] A {\End M}.$
  \end{enumerate}
\end{defn}

One can use the isomorphism $\hhcn [l-1] A {\End M}_{0} \cong \Hom {}
{(\Pi A)^{\otimes l-1} \otimes M} M_{-1}$ to write an \A n-module structure
$p_{M} = (p_{M}^{l})$ as a sequence of maps $m^{l}: (\Pi A)^{\otimes
  l-1} \otimes M \to M,$ with $1 \leq l \leq n,$ giving an idea why
this is called an \A n-module structure.

It will be helpful to expand $\hhcln A B$ by allowing
tensor degree zero elements. Set $\hhclnu A B = 
\prod_{0 \leq l \leq n} \hhcn [l] A B =  \Pi B \oplus \hhcln A B.$ Given an \A n-module structure $p_{M} \in \hhcln [n-1] A {\End M}_{0}$
on a complex $(M, \delta_{M})$, we consider
$\delta_{M} + p_{M} \in \hhcu A {\End M}_{0}$. Conversely, we can view
 \A n-module structures as elements $p_{M} \in \hhclnu A {\End M}_{0}$
such that $p^{0}_{M} \in \End M$ is a differential and $p^{\geq 1}_{M}$ is an
morphism of \A {n-1}-algebras from $A$ to the endomorphism \Ai-algebra of the complex
$(M, p^{0}_{M}).$ 

To define morphisms of \A n-modules, we add to the list of products
in \ref{defn-of-three-prods-on-hhc}.
\begin{defn} Let $A, M, N, P$ be graded modules.
  \begin{enumerate}
  \item Define the $\star$ product as follows,
    \begin{gather*}
      \star: \hhcu A {\Hom {} N P}_{k} \otimes \hhcu A {\Hom {} M
        N}_{l} \to \hhcu A {\Hom
        {} M P}_{k + l -1 }\\
      \alpha \otimes \beta = (\alpha^n) \otimes (\beta^n) \mapsto
      \left (s \gamma (s^{-1} \otimes s^{-1}) \sum_{j = 0}^{n}
        \alpha^j\otimes \beta^{n-j}\right) = \alpha \star \beta,
    \end{gather*}
    where $\gamma$ is the composition map.

   \item Let $(A, \nuln [n-1])$ be a nonunital \A {n-1}-algebra, for some
    $n \in \bN \cup \{ \infty \}$. An  element
    $f \in \hhclnu [n-1] A {\Hom {} M N}_{1}$ is a \emph{morphism of \A n-modules
      $(M, p_{M}) \to (N, p_{N})$} if
    $$p_{N} \star f + (f^{\geq 1})\circ \mu + f \star p_{M} = 0.$$
    The composition with a second morphism,
    $\t f \in \hhcu {A} {\Hom {} N P}_{1},$ is
    $\t f \star f \in \hhcu A {\Hom {} M N}_{1}.$
    \item An element
    $r \in \hhclnu [n-1] A {\Hom {} M N}_{2}$ is a \emph{homotopy
      between morphisms of \A n-modules
      $f, g: (M, p_{M}) \to (N, p_{N})$} if
   $$f - g = r \circ \nuln [n-1] + p_{N} \star r + r \star p_{M}.$$
  \end{enumerate}
\end{defn}

The unital version of \A n-modules is the following:
\begin{defn}
  Let $(A, \nuln [n-1])$ be a strictly unital \A {n-1}-algebra. An \A n-module $(M, p_{M})$ is strictly unital if $p_{M}$ is a
    strictly unital morphism of \A {n-1}-algebras.  A morphism of strictly unital
\A n-modules $(M, p_{M})\to (N, p_{N})$ is a morphism of
\A n-modules $f^{\leq n-1}$ such that $f^{\leq n-1}$ is in $\hhclnu [n-1] {\o A}
{\Hom {}M N}_{1},$ i.e., for all $1 \leq j \leq n$ and $a_{1}, \ldots,
a_{j-1} \in A,$ we have
$$f^{j}([a_{1}|\ldots|a_{i-1}|1|a_{i}|\ldots|a_{j-1}]) = 0.$$
A homotopy between morphisms of strictly unital \A n-modules $f^{\leq
  n}, g^{\leq n}: (M, p_{M})\to (N, p_{N})$ is a homotopy between
morphisms of \A n-modules $r^{\leq n-1}$ such that $r^{\leq n-1}$ is in $\hhclnu [n-1] {\o A}
{\Hom {}M N}_{2}$.
\end{defn}

To transfer strictly unital \Ai-structures, we need to place a further assumption on the pair $(A, 1).$

\begin{defn}\label{defn:split-elt}
  A \emph{split element} of a graded module $A$ is an element that
  generates a rank one free module. A \emph{graded module with split
    element} is a pair $(A, 1)$ with $1$ a split element in $A_{0}$, and a fixed (unlabeled) splitting $A \to k$ of the
  inclusion $k \to A, 1 \mapsto 1.$ Morphisms of graded modules with
  split elements $(A, 1) \to (B, 1)$ are always assumed to preserve
  the fixed splittings. An \emph{\A n algebra with split
    unit} is a triple $(A,1, \nuln)$, where $(A, 1)$ is a
  graded module with split element and $\nuln$ is an \A n
  algebra structure on $A$ that is strictly unital with respect to
  $1$.

  If $(A, 1)$ is a graded module with split element, we set $\o A =
A/(k \cdot 1),$ and consider it as a submodule of $A$ via the fixed splitting of $1.$ The projection $A \to \o
A$ identifies $\hhcln {\o A} A$ as a submodule of $\hhcln A A.$ The \emph{trivial
   strictly unital \Ai-structure on $(A,1)$} is a strictly unital
 element, denoted $\msu^{2} \in \hhcn [2] A A_{-1},$ such that for $n
 \geq 2,$ every strictly unital element
 $\nuln \in \hhcln A A_{-1}$ can be written $\nuln = \muln
 + \msu$ for a unique $\muln \in \hhcln {\o A} A_{-1};$ see
 \cite[Definition 4.3, Lemma 4.4]{SUal} for details. Analogously, there is a \emph{trivial
   strictly unital morphism} $\gsu \in \hhcn [1] A B_{0}$ and
 every strictly unital element $\aln \in \hhcln A B_{0}$ can be written uniquely as $\aln
 = \bln + \gsu,$ for $\bln \in \hhcn {\o A} B_{0}.$ It is
 clear that an element $r \in \hhcn A B_{1}$ is strictly unital if and
 only if $r$ is in $\hhcn {\o A} B_{1}.$
\end{defn}

\begin{rems}
  If $k$ is a field, then every strictly unital \A n algebra is an \A
  n algebra with
  split unit, since every element of $A$ is split. This is no longer
  the case if $k$ is not a field: if $I$ is a nonzero ideal of $k$,
  then $k/I$ has a strict, but not a split, unit.
\end{rems}

We
need the following homological algebra of complexes.

\begin{defn}\label{defn:semiproj-cx}
Let $(P, d_{P})$ and $(A, d_{A})$ be complexes, and recall the Hom-complex
$(\Hom {} P A, \dhom)$ has differential $\dhom(f) = d_{A} f -
(-1)^{|f|} f d_{P}.$ If $f: (A, d_{A}) \to (B, d_{B})$ is a morphism
of complexes, we have the following morphisms of
complexes,
\begin{align*}
  f_{*} =& \Hom {} P f: (\Hom {} P A, \dhom) \to (\Hom {} P B, \dhom)\\
  f^{*} =& \Hom {} f P: (\Hom {} B P, \dhom) \to (\Hom {} A P, \dhom).
\end{align*}
Thus $\Hom {} P -$ and $\Hom {} - P$ are endofunctors of the category
of complexes.

A complex $(P, d_{P})$ is \emph{semiprojective} if for every
   surjective quasi-isomorphism $f$, the morphism
   $f_{*} = \Hom {} P f$ is also a surjective quasi-isomorphism. A \emph{semiprojective resolution} of a complex $M$ is a
  quasi-isomorphism $P \xra{\simeq} M,$ with $P$ semiprojective.
\end{defn}

  \begin{rems}\label{chk:semiprojs}
If $k$ is a field, then every complex is semiprojective. If $P_n = 0$
for all $n \ll 0$, then $P$ is semiprojective if and only if each $P_n$ is
a projective $k$-module. In particular, if $M$ is
concentrated in degree 0 (i.e., $M$ is a module), then a projective resolution of $M$ is a semiprojective resolution. If $P$ is semiprojective, then $P_{n}$ is a projective $k$-module for
all $n,$ but not every complex of projective modules is semiprojective; for
instance, $ \ldots \to \bZ/4\bZ \xra{2} \bZ/4\bZ \xra{2}\bZ/4\bZ \to \ldots,$  where
$k = \bZ/4\bZ$. Every complex has a surjective semiprojective resolution.
    
  Semiprojective complexes are the cofibrant objects in the projective
  model structure on the category of $k$-complexes
  \cite[2.3]{MR1650134}.  Semifree complexes, of which
  semiprojective are summands, were first defined in
  \cite{AvHa86}. Semiprojective complexes are the K-projective
  complexes of projectives, using the terminology of \cite{Sp88}, and
  cell $k$-modules, using the terminology of
  \cite{MR1361938}.
\end{rems}

\section{Statement of results}
\begin{thm}\label{thm:main-transfer-result} Let $(B, \nu_{B})$ be a strictly unital \Ai-algebra.
  \begin{enumerate}
  \item Let $q:
    (\Pi A, \nu_{A}^{1})
    \to (\Pi B, \nu^{1}_{B})$ be a surjective semiprojective
    resolution of the complex underlying $(B, \nu_B)$, and assume that $A$ has a split element $1 \in A_{0}$
    with $\nu_{A}^{1}([1]) = 0$ and $q([1]) = [1].$ There exists
    $\nu_{A} \in \hhc A A_{-1},$ extending $\nu^{1}_{A},$ such that
    $(A, 1, \nu_{A})$ is an \Ai-algebra with split unit and $q$ is a
    strict morphism.\footnote{
A morphism of \Ai-algebras $f = (f^{l}) \in \hhc A B_{0}$ is \emph{strict} if
$f^{l} = 0$ for all $l \geq 2.$} If $(B, \nu_{B})$ is augmented, then
we can choose $\nu_{A}$ such that $(A, \nu_{A})$ is augmented.\footnote{A
  strictly unital \Ai-algebra $(B, \nu_{B})$ is \emph{augmented} if
  there is a strictly unital morphism $(B, \nu_{B}) \to (k, \msu)$.}
    
\item Consider a diagram of strictly unital \Ai-algebras,
\begin{displaymath}
    \begin{tikzcd}
      & (A, \nu_{A}) \ar[d, "q"]\\
      (C, \nu_{C}) \ar[r, "\alpha"'] & (B, \nu_{B}).
    \end{tikzcd}
  \end{displaymath}
  Assume that $(C, 1, \nu_{C})$ is an \Ai-algebra with split unit such
  that $(\Pi C, \nu_{C}^{1})$ is a semiprojective complex, and that $q
  = q^{1}$ is strict, with $q^{1}: (\Pi A, \nu_A^{1}) \to (\Pi B,
  \nu_B^{1})$ a surjective
  quasi-isomorphism of complexes. Assume further that there is a
  morphism of chain complexes $\delta^{1}: (\Pi C, \nu^{1}_{C}) \to
  (\Pi A, \nu^{1}_{A})$ such that $\delta^{1}[1] = [1]$ and $q^{1}
  \delta^{1} = \alpha^{1}.$ Then there exists a strictly
  unital morphism of \Ai-algebras $\delta: (C, \nu_{C}) \to (A, \nu_A)$ such that
  $q \delta = \alpha.$

  \item If $H_{0}(\hhcn [n+1] {\o C} A, \dab [\overbar{C}A])= 0$ for all $n \geq 1,$ and $\delta'$ is
  another lifting of $\alpha$ through $q$, then $\delta$ and $\delta'$ are homotopic by a strictly unital homotopy:
  \begin{displaymath}
    \begin{tikzcd}
      & A \ar[d]\\
      C \ar[ur, dotted, "\delta", shift left = 1ex] \ar[ur, dotted,
      "\delta'"', "{\rotatebox[origin=c]{45}{$\sim$}}", shift left =
      -1ex, outer sep = -.5ex] \ar[r] & B.
    \end{tikzcd}
  \end{displaymath}
  In particular, if $A$ satisfies $H_{0}(\hhcn [n+1] {\o A} A, \dab
  [\overbar{A}A])= 0$ for all $n \geq 1,$ then any two elements $\nu_{A}, \t \nu_{A} \in \hhc A A,$ such that $(A, 1,
  \nu_{A})$ and $(A, 1, \t \nu_{A})$ are \Ai-algebras with split units
  and $q$ is a strict morphism, are homotopy equivalent (via strictly
  unital homotopies).
  \end{enumerate}
\end{thm}

\begin{rems}
  In part 2 above, there is always a morphism of chain complexes
  $\delta^{1}: (\Pi C, \nu^{1}_{C}) \to (\Pi A, \nu^{1}_{A})$ such
  that $q^{1}\delta^{1} = \alpha^{1},$ using
  the lifting properties of the semiprojective complex $(\Pi C,
  \nu^{1}_{C}).$ But it is not clear that $\delta^{1}$ can be
  chosen such that $\delta^{1}[1] = [1].$ In certain situations such a
  $\delta^{1}$ always exists, e.g., if $A_{0}, B_{0}, C_{0}$ are all
  cyclic $k$-modules. It is also often the case, as in the corollary below, that $A = C$ and
  $\delta^{1}$ can be chosen to be the identity.
\end{rems}

\begin{cor}\label{cor:transfer-for-classical-algs}
Let $B$ be an associative $k$-algebra, and let $\pi: (A, d_{A}) \to B$
be a $k$-projective resolution. Assume that $A$ has a split
element $1 \in A_{0}$ with $\pi(1) = 1 \in B.$ Then $A$ has a strictly
unital \Ai-algebra structure such that $s\pi s^{-1}$ is a
strict morphism of \Ai-algebras, and this is unique up to strictly unital homotopy.
\end{cor}

\begin{proof}[Proof of Corollary \ref{cor:transfer-for-classical-algs}]
We can consider $B$ as an \Ai-algebra $(B, \nu_{B})$ with $\nu^{n}_{B}
= 0$ for all $n \neq 2.$ The complex $(A, d_{A})$ is semiprojective,
 so by \ref{thm:main-transfer-result}.(1)
$\nu_{A}$ exists as claimed. To see the uniqueness statement, assume
that $\nu_{A}$ and $\t \nu_{A}$ both satisfy the hypothesis. By
\ref{thm:main-transfer-result}.(2), with $\delta^{1} = 1,$ there are strictly unital
morphisms $\delta, \epsilon$ with $q \epsilon = q$ and $q \delta = q,$
where $q = s \pi s^{-1},$
\begin{displaymath}
    \begin{tikzcd}
      & A \ar[d, "q"] \ar[dl, dotted,
      "\epsilon", shift left =
      1ex]\\
      A \ar[ur, dotted, "\delta", shift left = 1ex]  \ar[r, "q"'] & B.
    \end{tikzcd}
  \end{displaymath}
It follows that $q \epsilon \delta = q = q \delta \epsilon.$ Now note the canonical
map $(\hhcn [n+1] {\o A}
A, \dab [\overbar{A}A]) \to (\hhcn [n+1] {\o A} B, \dab
[\overbar{A}B])$ is a quasi-isomorphism, since $\o A$ is
semiprojective. Further, $\hhcn [n+1] {\o A} B_{0} \cong \Hom {} {(\o
  A^{\otimes n+1})_{-n}} B = 0$ for $n \geq 1,$ since
$\o A^{\otimes n + 1}$ is concentrated in nonnegative degrees. Thus $H_{0}(\hhcn [n+1] {\o A}
A, \dab [\overbar{A}A])= 0$ for all $n \geq 1,$ and applying
\ref{thm:main-transfer-result}.(3) shows that $\epsilon \delta$ and
$\delta \epsilon$ are homotopic to the identity (the
definition of $(A, \nu_{A})$ and $(A, \t \nu_{A})$ being homotopy equivalent).
\end{proof}

\begin{thm}\label{thm:transfer-of-mod-strs}
Let $(A, 1, \nu)$ be an \Ai-algebra with split unit such that
$(\Pi A, \nu^{1})$ is a semiprojective complex, and let $(M,
p_{M})$ be a strictly unital \Ai-module over $(A, 1, \nu).$
\begin{enumerate}
\item Let $q: (G, p^{0}_{G}) \to (M, p^{0}_{M})$ be a surjective
  semiprojective resolution of the complex underlying $(M, p_{M})$. There exists $p_{G} \in \hhc
  A {\End G}_{0}$, extending $p^{0}_{G},$ such that $(G, p_{G})$ is a
  strictly unital \Ai-module over $(A, 1, \nu_{A})$ and $q$ is a
  strict morphism of \Ai-modules.\footnote{A morphism of \Ai-modules $f = (f^{n})$ is
    strict if $f^{n} = 0$ for all $n \geq 1$.}
  
\item Consider a diagram of strictly unital \Ai-modules over $(A, 1,
  \nu_{A}):$
  \begin{displaymath}
    \begin{tikzcd}
      & (G, p_{G}) \ar[d, "q"]\\
      (N, p_{N}) \ar[r, "\alpha"'] & (M, p_{M}).
    \end{tikzcd}
  \end{displaymath}
  Assume that $(N, p_{N}^{0})$ is a semiprojective complex, and that $q = q^{1}$ is strict, with $q^{1}: (G, p_G^{0}) \to (M,
  p_M^{0})$ a surjective
  quasi-isomorphism of complexes. Then there exists a strictly
  unital morphism of \Ai-modules $\delta: (N, p_{N}) \to (G, p_G)$ such that
  $q \delta = \alpha.$
  
\item If $H_{1}(\hhcn {\o A} {\Hom {} N G}, \dab
[\overbar{A}\oHom]) = 0$ for all $n \geq 1,$ and $\delta'$ is
  another lifting of $\alpha$ through $q$, then the strictly unital
  morphisms $\delta$ and $\delta'$ are homotopic (via a strictly
  unital homotopy):
  \begin{displaymath}
    \begin{tikzcd}
      & G \ar[d]\\
      N \ar[ur, dotted, "\delta", shift left = 1ex] \ar[ur, dotted,
      "\delta'"', "{\rotatebox[origin=c]{45}{$\sim$}}", shift left =
      -1ex, outer sep = -.5ex] \ar[r] & M.
    \end{tikzcd}
  \end{displaymath}
  In particular, if $G$ satisfies $H_{1}(\hhcn {\o A} {\Hom {} G G},
  \dab [\overbar{A}\oHom])= 0$ for all $n \geq 1,$ then any two
  strictly unital \Ai-module structures
 $p_{G}, \t p_{G} \in \hhc A {\End G}_{0},$ such that $q$ is a strict
 morphism, are homotopy equivalent (via strictly unital homotopies).
\end{enumerate}
\end{thm}

\begin{cor}
Let $B$ be an associative $k$-algebra, and $\pi: (A, d_{A}) \to B$
 a $k$-projective resolution such that $A$ has a split
element $1 \in A_{0}$ with $\pi(1) = 1 \in B.$ Let $\nu_{A}$ be a
strictly unital \Ai-structure on $A$ such that $s\pi s^{-1}$ is a
strict morphism.  (Such $\nu_{A}$ exists by
\ref{cor:transfer-for-classical-algs}.) Let $M$ be a $B$-module and $G
\xra{\simeq} M$ a $k$-projective resolution. Then $G$ has a strictly
unital \Ai-module structure over $(A, 1, \nu_{A}),$ and this
is unique up to strictly unital homotopy.
\end{cor}

The proof is similar to the proof of \ref{cor:transfer-for-classical-algs}.

\section{Obstruction theory}
The main tool used in the proofs of Theorems
\ref{thm:main-transfer-result} and \ref{thm:transfer-of-mod-strs} is
obstruction theory. This is a way of extending an \A n-object (algebra,
morphism, or homotopy) to an \A {n+1}-object, towards the goal of
building an \Ai-object. This general strategy is based on the following.
\begin{lem}\label{lem:ainf-iff-analln}Let $A$ and $B$ be graded modules.
  \begin{enumerate}
  \item An element $\nu \in \hhc A A_{-1}$ is a nonunital \Ai-algebra
    structure if and only if, for all $n \geq 1$,
    $\nuln \in \hhcln A A_{-1}$ is a nonunital \A n-algebra
    structure.
    
\item An element $\alpha \in \hhc A B_{0}$ is a morphism
  of nonunital \Ai-algebras $(A, \nu_{A}) \to (B, \nu_{B}),$ if and only if, for all
  $n \geq 1$, $\aln \in \hhcln A B_{0}$ is
  a morphism of \A n-algebras $(A, \nuln_{A}) \to (B, \nuln_{B}).$ 

\item An element $r \in \hhc A B_{1}$
  is a homotopy between morphisms $\alpha, \beta: (A, \nu_{A})
  \to (B, \nu_{B})$ if and only if, for all
  $n \geq 1,$ $\rln \in \hhcln A B_{1}$ is a homotopy between $\aln$ and $\bln$.
    \end{enumerate}
\end{lem}

\begin{proof}
Note first that $\left (\nu \circ \nu \right )^{\sleq n} = \left (\nuln \circ
  \nuln \right )^{\sleq n}.$ Thus, if $\nu \circ \nu = 0,$ then 
$\left (\nu \circ \nu  \right )^{\sleq n} = \left (\nuln \circ \nuln
\right )^{\sleq n} = 0$ for all $n \geq 1,$ so
$\nuln$ is an \A n-algebra structure for all $n \geq 1$. Conversely, if $\nuln$ is an \A
n-algebra for all $n \geq 1$, then $\left (\nu \circ \nu \right
)^{\sleq n} = 0$  for
all $n \geq 1$, and so $\nu \circ \nu = 0.$  This proves part 1, and the other parts are proved in an analogous way. 
\end{proof}

To pass from an \A n-object to \A {n+1}-object requires one to show
that a certain cycle, called the obstruction, is a boundary. The complex where this occurs
is the following.
\begin{defn}\label{defn:diff-dab}
  Let $\nu_{A}^{1}$ and $\nu_{B}^{1}$ be \A 1-algebra structures on
  $A$ and $B$ (i.e., $(\Pi A, \nu^{1}_{A})$ and $(\Pi B, \nu^{1}_{B})$
  are complexes). For any $n \in \bN \cup \{ \infty \},$ we consider
  the complex $(\hhcn [n+1] A B, \dab),$ where
$\dab$ is the Hom-complex differential
      $\dhom$ on $\hhcn [n+1] A B = \Hom
    {} {(\Pi A)^{\otimes n}} {\Pi B}$ between the complexes $((\Pi
    A)^{\otimes n}, \delta_{\otimes} = \sum_{j} 1^{\otimes j}
    \otimes \nu^{1}_{A}\otimes 1^{n - j - 1})$ and $(\Pi B,
    \nu^{1}_{B})).$ Note that $\dab(\alpha) = \nu^{1}_{B} \alpha - (-1)^{|\alpha|} \alpha \circ
      \nu_{A}^{1}.$
  \end{defn}

\begin{defn} Let $A$ and $B$ be graded modules.
  \begin{enumerate}
  \item If $(A, \nuln)$ is an \A n-algebra, an element
    $\nu^{n+1} \in \hhcn [n+1] {A} A_{-1}$ \emph{extends $\nuln$} if
    $(A,\nuln [n+1] = \nuln + \nu^{n+1})$ is an \A {n+1}-algebra. The
    \emph{obstruction (to extending $\nuln$}) is
$$\obs {\nuln} = -\left (\nuln \circ \nuln \right )^{\sleq n+1} \in
\hhcn [n+1] {A} A_{-2}.$$

\item   If $(A, \nuln[n+1]_{A})$ and $(B, \nuln[n+1]_{B})$ are
  nonunital \A
  {n+1}-algebras, and $\aln \in \hhcln A B_{0}$ is a morphism of
 nonunital \A n-algebras $(A,
 \nuln_{A}) \to (B, \nuln_{B})$, an element $\alpha^{n+1} \in \hhcn [n+1] {A} B_{0}$
 \emph{extends $\aln$} if $\aln [n+1]$ is a morphism of \A n-algebras $(A, \nuln [n+1]_{A}) \to (B, \nuln
 [n+1]_{B})$. The \emph{obstruction} (to extending $\aln$ to a
 morphism of \A {n+1}-algebras) is
 \[\obs {\aln} = -\piln [n+1] {\nuln [n+1]_{B}* \aln - \aln \circ
     \nuln [n+1]_{A}} \in \hhcn [n+1] {A} B_{-1}.\]
 
\item  If $\aln [n+1], \bln [n+1]: (A, \nuln [n+1]_{A}) \to (B, \nuln
 [n+1]_{B})$ are morphisms of nonunital \A {n+1}-algebras, and $\rln \in \hhcln A
 B_{1}$ a homotopy between $\aln$ and $\bln,$ an element $r^{n+1} \in \hhcn [n+1] A B_{1}$
 \emph{extends $\rln$} if $\rln [n+1]$ is a homotopy between $\aln
 [n+1]$ and $\bln [n+1].$ The \emph{obstruction}
 (to extending $\rln$) is 
\[\obs{\rln} = \alpha^{n+1} - \beta^{n+1} -
 \piln [n+1]{\nuln[n+1]_{B} \hsp \rln - \rln \circ \nuln[n+1]_{A}} \in
 \hhcn [n+1] A B_{0}.\] 
\end{enumerate}
\end{defn}

The following is stated in \cite[B.1]{LH} and a proof of the first
part is given. We give full proofs of all three parts below for the ease of
the reader, because they are essential to what follows, and because
\loccit implicitly assumes that $k$ is a field (rather that every
module is semisimple), though this hypothesis is not used there. Our
proofs are based on the proof of part 1 given in \loccit

\begin{prop}\label{prop:obs-are-cycles-for-algebras} Let $A$ and $B$
  be graded modules.
  \begin{enumerate}
  \item If $(A, \nuln)$ is a nonunital \A n-algebra, then the obstruction
    $\obs{\nuln}$ is a cycle in ($\hhcn [n+1] {A} A, \dab [AA]$). An element $\nu^{n+1}$ in
    $\hhcn [n+1] {A} A_{-1}$ extends $\nuln$ if and only if
    $ \dab [AA](\nu^{n+1}) = \obs {\nuln}.$
    
\item \label{prop:obs-cycles-morphisms}
  If $(A, \nuln[n+1]_{A})$ and $(B, \nuln[n+1]_{B})$ are
  nonunital \A
  {n+1}-algebras, and $\aln \in \hhcln A B_{0}$ is a morphism of
 nonunital \A n-algebras $(A,
 \nuln_{A}) \to (B, \nuln_{B})$, then the obstruction $\obs {\aln}$ is
 a cycle in $(\hhcn [n+1] {A} B, \dab)$. An element
 $\alpha^{n+1} \in \hhcn [n+1] {A} B_{0}$ extends $\aln$ if and
 only if $\dab(\alpha^{n+1}) = \obs{\aln}.$

 \item   If $\aln [n+1], \bln [n+1]: (A, \nuln [n+1]_{A}) \to (B, \nuln
 [n+1]_{B})$ are morphisms of nonunital \A {n+1}-algebras, and $\rln \in \hhcln A
 B_{1}$ is a homotopy between $\aln$ and $\bln,$ then the obstruction
 $\obs{\rln}$ is a cycle in $(\hhcn [n+1] A B, \dab).$ An
 element $r^{n+1}$ extends $\rln$ if and only if $\dab(r^{n+1}) = \obs{\rln}.$
  \end{enumerate}
\end{prop}

For the proofs we will need the following technical material. If $V$
is a graded module and $n \in \bN \cup \{ \infty \}$, we set $\Tcon V
= \bigoplus_{1 \leq i \leq n} V^{\otimes i}$ to be the
truncated tensor coalgebra on $V$. This is a nonunital
graded coalgebra, with comultiplication the linear extension of
$\Delta(v_1 \otimes \ldots \otimes v_{i}) = \sum_{1 \leq j \leq i-1}
(v_1 \otimes \ldots \otimes v_{j}) \otimes (v_{j+1} \otimes \ldots
\otimes v_{i}).$ Note that $\hhcn A B = \Hom {} {\Tcon {\Pi A}} {\Pi B}.$

If $C, D$ are graded coalgebras and $\alpha, \beta: C \to D$ graded
coalgebra morphisms, an \emph{$(\alpha,\beta)$-coderivation} is a degree $-1$ map $r: C \to D$ satisfying
$\Delta_{D} r - (r \otimes \alpha + r \otimes \beta) \Delta_{C} = 0.$ We write
  $\Coderab(C,D)$ for the set of $(\alpha,\beta)$-coderivations. It will often
  be the case that $\alpha = 1_{C} = \beta;$ a \emph{coderivation} is a
  $(1_{C},1_{C})$-coderivation, and we will write $\Coder(C,C)$ for the set of such.

\begin{lem}\label{lem:isom_hhc_coder}
Let $A$ and $B$ be graded modules and fix $n \in \bN \cup \{ \infty \}.$
  \begin{enumerate}
          \item The canonical projection $\pi_1: \Tcon {\Pi B} \to \Pi B$
    induces an isomorphism,
    \[ \Psi_{n} = (\pi_1)_*: \Hom {\Coalgk} {\Tcon {\Pi A}} {\Tcon {\Pi B}}
      \xra{\cong} \Hom {} {\Tcon {\Pi A}} {\Pi B}_{0} = \hhcln A B_{0}.\] The inverse applied to $\alpha =(\alpha^l) \in \hhcln A B_{0}$ is
    given by:
    \[ \pi_j \Psi^{-1}_{n}(\alpha)|_{(\Pi A)^{\otimes k}} = \sum_{i_1 + \ldots +
        i_j = k} \alpha^{i_1}\otimes \ldots \otimes \alpha^{i_{j}}.\]
  \item Let $\Psi_{n}^{-1}(\alpha),\Psi_{n}^{-1}(\beta): \Tcon {\Pi A} \to
    \Tcon {\Pi B}$ be two
    coalgebra morphisms, with $\alpha, \beta \in \hhcln A B_{0}.$ The canonical projection $\pi_1: \Tcon {\Pi B} \to \Pi B$
    induces an isomorphism,
    \[ \Phi_{n}^{\alpha,\beta} = (\pi_1)_*: {\Coderab}( \Tcon {\Pi A},
      \Tcon {\Pi B}) \xra{\cong} \Hom {} {\Tcon {\Pi A}} {\Pi B} =
      \hhcln
      A B.\]
    The inverse applied to
    $r = (r^{l}) \in \hhcln A B$ is given by
    \begin{align*}
      \pi_{j} (\Phi^{\alpha,\beta}_{n})^{-1}(r)|_{(\Pi A)^{\otimes k}}
      &= \sum_{\substack{i_{1}+\ldots+i_{j} = k\\ 1 \leq m \leq j}} \alpha^{i_{1}}
        \otimes \ldots \otimes \alpha^{i_{m-1}} \otimes r^{i_{m}} \otimes
        \beta^{i_{m+1}} \otimes \ldots \otimes \beta^{i_{j}}.
                                        \end{align*}
        \end{enumerate}
      \end{lem}

      \begin{proof}
        Part 1 is \cite[2.19]{MR2844537}, and Part 2, in case both
        morphisms are the identity, is \cite[2.16]{MR2844537}. The
        proof given in \loccit is easily modified to prove part 2 for
        arbitrary $\alpha$ and $\beta$.
      \end{proof}

The isomorphisms above are related
to the products defined in \ref{defn-of-three-prods-on-hhc} as
follows. (We write
$\Phi_{n}$ for $\Phi_{n}^{1_{C}, 1_{C}}$ below and in the sequel.)

\begin{lem}\label{lem:gerst-and-star-prod-in-terms-of-can-isoms}
Let $A$ and $B$ be graded modules and fix $n \in \bN \cup \{ \infty \}.$
\begin{enumerate}
\item For $\mu \in \hhcln A B$ and $\nu \in \hhcln A A,$ we have 
$\left (\mu \circ \nu \right )^{\sleq n} = \mu \Phi_{n}^{-1}(\nu).$
\item For $\nu \in \hhcln B B$ and $\alpha \in \hhcln A B_{0}$, we
  have $\left (
\nu * \alpha \right )^{\sleq n} = \nu \Psi_{n}^{-1}(\alpha).$
\item For $\nu \in \hhcln B B_{-1}, r \in \hhcln A B_{1},$ and
  $\alpha, \beta \in \hhcln A B_{0},$ we have $\nu \hsp r = \nu (\Phi_{n}^{\alpha,\beta})^{-1}(r).$
\end{enumerate}
\end{lem}

\begin{proof}
Note that $\mu \Phi_{n}^{-1}(\nu)$ is a morphism $\Tcon {\Pi
A} \to \Pi B,$ i.e., is an element of $\hhcln A B.$ By the definition
of $\Phi^{-1},$ this is given in tensor degree $1 \leq j \leq n$ by
$ \sum_{\substack{1 \leq i \leq k\\0 \leq j
    \leq i-1}}  \mu^{i}
(1^{\otimes j} \otimes \nu^{k-i+1} \otimes 1^{\otimes i - j - 1}),$ which agrees with $\mu \circ \nu$ in
tensor degree $j$. This proves part 1 and
the others are proved in an analogous way.
\end{proof}

\begin{proof}[Proof of Proposition \ref{prop:obs-are-cycles-for-algebras}]
We first prove part 1. Let $\nu^{n+1} \in \hhcn [n+1] A A_{-1}$ be an
arbitrary element, and set $\nuln [n+1] = \nuln + \nu^{n+1}.$ We
have,
\begin{equation}\label{eq:obs}
  \begin{aligned}
    \left (\nuln [n+1] \circ \nuln [n+1] \right )^{\sleq n+1} &= 
    \left (\nuln \circ \nuln \right )^{\sleq n+1} + \nu^{1} \circ \nu^{n+1} + \nu^{n+1}\circ
    \nu^{1}\\
    &= -\obs{\nuln} + \dab [AA](\nu^{n+1}).
  \end{aligned}
\end{equation}
Thus to show $\obs{\nuln}$ is a
cycle, it is enough to show $\left(\nuln [n+1] \circ \nuln [n+1]\right)^{\sleq n+1}$
is a cycle.

Since $\nu^{1}$ has tensor degree 1, $\piln [n+1] {\nuln [n+1]
\circ \nuln [n+1]} \circ \nu^{1} =  \piln [n+1]{(\nuln [n+1] \circ \nuln [n+1]) \circ
\nu^{1}}.$ Since $\nuln [n+1] \circ \nuln [n+1]$ is concentrated in tensor
degrees at least $n+1$ (because $\nuln$ is an \A n-algebra structure), $\piln [n+1]{(\nuln [n+1] \circ \nuln [n+1]) \circ
\nu^{1}} = \piln [n+1]{(\nuln [n+1] \circ \nuln [n+1]) \circ 
\nuln [n+1]}.$ Analogously, $\nu^{1} \circ \piln [n+1] {\nuln [n+1] \circ \nuln [n+1]} =
\piln [n+1] { \nuln [n+1] \circ (\nuln [n+1] \circ \nuln [n+1])}$. Thus
\begin{gather*}
  \dab [AA]\left (\piln [n+1]{\nuln [n+1] \circ \nuln [n+1]}\right )  =\\
  \piln [n+1] {\nuln [n+1] \circ (\nuln [n+1] \circ \nuln [n+1]) - (\nuln [n+1] \circ \nuln [n+1]) \circ 
\nuln [n+1]} = 0,
\end{gather*}
since $f \circ
(f \circ f)- (f \circ f) \circ f = 0$  for any element of $f \in \hhc
A A_{-1}$ by \cite[\S 2]{MR0161898}, and thus $\obs{\nuln}$ is a
cycle. By definition, $\nu^{n+1}$ extends $\nuln$ if and only if $\piln
[n+1]{\nuln [n+1] \circ \nuln [n+1]} = 0,$
and by \eqref{eq:obs} this is equivalent to $\dab [AA] (\nu^{n+1})
= \obs{\nuln}.$

To prove part 2, let $\alpha^{n+1} \in \hhcn [n+1] A B_{0}$ be an
arbitrary element. We have
\begin{equation}
\piln [n+1] {\nuln [n+1]_B * \aln [n+1] - \aln [n+1]
\circ \nuln [n+1]_A} = -\obs{\aln} +
\dab(\alpha^{n+1}).\label{eq:obs-morphisms}
\end{equation}
Thus
to show $\obs{\aln}$ is a cycle,
it is enough to show \eqref{eq:obs-morphisms} is a cycle.
Set
\begin{align*}
d_{A} &=
\Phi^{-1}_{n+1}(\nuln [n+1]_A) \in \Coder(\Tcon [n+1]{\Pi A},
        \Tcon [n+1] {\Pi A})\\
  d_{B} &=
\Phi^{-1}_{n+1}(\nuln [n+1]_B) \in \Coder(\Tcon [n+1]{\Pi B},
          \Tcon [n+1] {\Pi B})\\
  \zeta &= \Psi^{-1}_{n+1}(\aln [n+1]):
\Tcon [n+1] {\Pi A} \to \Tcon [n+1] {\Pi B}.
\end{align*}
By Lemma
\ref{lem:gerst-and-star-prod-in-terms-of-can-isoms} we have
\begin{displaymath}
\piln [n+1] {\nuln [n+1]_B * \aln [n+1] - \aln [n+1]
\circ \nuln [n+1]_A} = \nuln [n+1]_B \zeta - \aln [n+1] d_{A},
\end{displaymath}
and thus, we aim to show $\dab(\nuln [n+1]_B \zeta - \aln [n+1] d_{A}) = 0.$

We first claim that $\nu^{1}_{B}(\nuln [n+1]_B \zeta -
\aln [n+1] d_{A}) = -\piln [n+1]{\nuln [n+1]_{B} \zeta
d_{A}}$. Indeed, because
$\nu^{1}_{B}$ has tensor degree 1, $\nu^{1}_{B}(\nuln [n+1]_B \zeta -
\aln [n+1] d_{A}) = \nu^{1}_{B}(d_{B} \zeta - \zeta d_{A})$, using the
definition of $\Phi^{-1}$ and $\Psi^{-1}$. Since $d_{B} \zeta -
\zeta d_{A}$ is concentrated in tensor degree at least $n+1$ (because
$\aln$ is a morphism of \A n-algebras), we have $\nu^{1}_{B}(d_{B} \zeta -
\zeta d_{A}) = \piln [n+1]{\nuln [n+1]_{B}(d_{B} \zeta -
  \zeta d_{A})}$, and the claim follows from
\[\nuln [n+1]_B
d_{B} = \piln [n+1]{\nuln [n+1]_{B} \circ \nuln [n+1]_{B}} = 0,\]
which holds since
$\nuln [n+1]_{B}$ is an \A {n+1}
algebra structure.

We now claim that $(\nuln [n+1]_B \zeta -
\aln [n+1] d_{A})\circ \nu^{1}_{A} = \piln [n+1]{\nuln [n+1]_{B}
\zeta d_{A}}.$ By tensor degree considerations $(\nuln [n+1]_B \zeta -
\aln [n+1] d_{A})\circ \nu^{1}_{A} = \piln [n+1]{(\nuln [n+1]_B \zeta -
\aln [n+1] d_{A})\circ \nuln [n+1]_{A}},$ and this is equal to
$\piln [n+1] {(\nuln [n+1]_B \zeta -
\aln [n+1] d_{A})d_{A}}$ by Lemma
\ref{lem:gerst-and-star-prod-in-terms-of-can-isoms}. The claim
follows since $d_{A} d_{A} = 0,$ since $\nu_{A}^{n+1}$ is an \A
{n+1}-algebra structure. Putting the two claims together, we have 
\[\dab(\nuln [n+1]_B \zeta -
\aln [n+1] d_{A}) 
= -\piln [n+1]{\nuln [n+1]_{B} \zeta
d_{A}} + \piln [n+1]{\nuln [n+1]_{B}
\zeta d_{A}} = 0.\]
This shows that $\obs{\alpha^{\leq n}}$ is a cycle. By definition, $\alpha^{n+1}$ is an extension of $\aln$ if and only if
\[\piln [n+1] {\nuln [n+1]_{B} * \aln [n+1] - \aln [n+1]
\circ \nuln [n+1]_{A}} = 0,\]
 and this is equivalent to
$\dab(\alpha^{n+1}) = \obs{\aln}$ by \eqref{eq:obs-morphisms}.

To prove part 3, let $r^{n+1} \in \hhcn [n+1] A B_{1}$ be an arbitrary
element. We have
\begin{displaymath}
 \piln [n+1]{\nuln[n+1]_{B} \hsp \rln[n+1] - \rln[n+1] \circ
   \nuln[n+1]_{A}} + \beta^{n+1} - \alpha^{n+1} = -\obs{\rln} + \dab(r^{n+1}).
\end{displaymath}
Thus, as above, we aim to show the left side of the above equation is a
cycle. Using similar techniques as in the proof of
part 2, one computes
\begin{gather*}
\nu^{1}_{B} \piln [n+1]{\nuln[n+1]_{B} \hsp \rln[n+1] - \rln[n+1] \circ
   \nuln[n+1]_{A}} = \piln [n+1]
 {-\nu^{1}_{B} \rln [n+1] \circ \nu_{A}^{n+1}}\\
 \piln [n+1]{\nuln[n+1]_{B} \hsp \rln[n+1] - \rln[n+1] \circ
   \nuln[n+1]_{A}} \circ \nu^{1}_{A} = \piln [n+1]
 {\nu^{1}_{B} \rln [n+1] \circ \nu_{A}^{n+1}}\\
 \alpha^{n+1} \circ \nu^{1}_{A} -
\nu^{1}_{B} \alpha^{n+1} = 0 \quad \quad \beta^{n+1} \circ \nu^{1}_{A} - \nu^{1}_{B} \beta^{n+1}  = 0
\end{gather*}
where the last two equalities use that $\alpha^{\leq n-1}$ and $\beta^{\leq n+1}$  are morphisms of \A
{n+1}-algebras. Combining the above four equations, we see that
\[\dab(\piln [n+1]{\nuln[n+1]_{B} \hsp \rln[n+1] - \rln[n+1] \circ
    \nuln[n+1]_{A}} + \beta^{n+1} - \alpha^{n+1}) = 0.\]
The rest of the proof is analogous to part 2.
\end{proof}

To use obstruction theory to construct strictly unital objects, we
need to assume $(A, 1)$ is a graded module with split element
(see Definition \ref{defn:split-elt}).
 It follows from Definition \ref{defn:strict-unit} that to extend a
 strictly unital \A n-algebra structure on $A$ to a strictly unital \A
 {n+1}-algebra structure, we need only an element of $\hhcn
 [n+1] {\o A} A$, not $\hhcn [n+1] A A$ (and analogously for
 morphisms and homotopies). If $\nu^{1}_{A}$ is strictly unital, i.e., $\nu^{1}_{A}[1] = 0,$
 the differential $\dab$ preserves $\hhcn [n+1] {\o A} B.$ 
 We write $\doab$ for $\dab|_{\hhcn [n+1]{\overbar A} B},$ so $(\hhcn
 [n+1] {\o A} B, \dab [\overbar{A}B])$ is a subcomplex of $(\hhcn
 [n+1] A B, \dab).$ If we
 denote by $\nu^{1}_{\overbar A}: \Pi \o A \to \Pi \o A$ the map induced
 by $\nu^{1}_{A}$, then $\doab(\alpha) = \nu^{1}_{B} \alpha - (-1)^{|\alpha|} \alpha
 \circ \nu^{1}_{\overbar A}.$ The
 following shows we can work with the complex
 $(\hhcn [n+1] {\o A} B, \doab)$ to do strictly unital obstruction theory.

\begin{lem}\label{lem:sual-obs-live-in-subcomplex}  Let $(A, 1)$ be a graded module with split element, $B$ a graded
  module with fixed element $1 \in B_{0},$ and fix $n \in \bN \cup \{
  \infty \}.$
   \begin{enumerate}
    \item  If $(A, 1, \nuln_{A} = \muln_{A} + \msu)$ is an \A {n}-algebra with
 split unit, where $\muln_{A} \in \hhcln {\o A} A,$ then the obstruction $\obs {\nuln_{A}}$ is in $\hhcn [n+1] {\o A} A \subseteq \hhcn
  [n+1] A A$.  For $\mu^{n+1}_{A} \in \hhcn [n+1] {\o A}
  A_{-1},$ $\nuln [n+1]_{A} = \muln [n+1]_{A} + \msu$ is a
  strictly unital \A {n+1}-algebra if  $\dab
  [\overbar{A}A](\mu^{n+1}_{A}) = \obs{\nuln_{A}},$ when $n \geq 2,$ or
  $\dab [\overbar{A}A](\mu^{2}_{A}) + \dab[AA](\msu^{2}) = 0$ when $n = 1.$

\item  If $(A, 1, {\nuln [n+1]_{A}})$ is an \A {n + 1}-algebra with
 split unit, $(B, \nuln [n+1]_{B})$ a strictly unital \A {n +
   1}-algebra, and $\aln = \bln +
 \gsu: (A,
 \nuln_{A}) \to (B, \nuln_{B})$ a strictly unital morphism of \A
 n-algebras, with $\bln \in \hhcln {\o A} B$, then the obstruction $\obs {\aln}$ is in $\hhcn [n+1] {\o A} B \subseteq \hhcn
 [n+1] A B$.  For $\beta^{n+1} \in \hhcn [n+1] {\o A}
  B_{0},$ $\aln [n+1] = \bln [n+1] + \gsu$ is a
  strictly unital morphism of \A {n+1}-algebras if $\doab(\beta^{n+1}) = \obs{\aln}.$
 
\item If $\aln [n+1],
  \bln[n+1]: (A, \nuln[n+1]_{A}) \to (B, \nuln[n+1]_{B})$ are strictly
  unital morphisms of strictly unital \A {n+1}-algebras, and
  $\rln \in \hhcln {\o A}B_{1}$ a strictly unital homotopy between
  $\aln$ and $\bln,$ then the obstruction $\obs{\rln}$ is in $\hhcn [n+1]
  {\o A} B.$ For $r^{n+1} \in \hhcn [n+1] {\o A}
  B_{1},$ $\rln [n+1]$ is a
  strictly unital homotopy between $\aln [n+1]$ and $\bln [n+1]$ if $\doab(r^{n+1}) = \obs{\rln}.$
 \end{enumerate} 
\end{lem}

\begin{proof}
We prove part 1. By definition, $\obs{\nuln} = \piln [n+1]
{\nuln \circ \nuln}.$ If $n= 1,$ then $\obs{\nuln} = 0 \in \hhcn [2]
{\o A} A.$ If $n \geq 2,$ we can write $\nuln = \muln + \msu,$ for
$\muln \in \hhcln {\o A} A.$ Write $\muln = \omuln + \hln \in \hhcln {\o A}
      {\o A} \oplus \hhcln {\o A} k,$ so $\nuln \circ
      \nuln = (\omln + \hln + \msu) \circ (\omln + \hln + \msu)$. It
      is clear that $\omuln \circ \hln = 0 = \hln \circ \hln,$ and $\msu \circ \msu = 0$
      since it is an \Ai-algebra structure. By \cite[Lemma
      4.7.1]{SUal}, $\omuln \circ \msu + \msu \circ \omuln = 0,$ and
      by \cite[Lemma 4.7.2]{SUal}, $\hln \circ \msu + \msu \circ \hln
      = \msu(\hln \otimes 1 + 1 \otimes \hln).$ Thus, $\nuln \circ
      \nuln = \omuln \circ \omuln + \hln \circ \omuln + \msu(\hln
      \otimes 1 + 1 \otimes \hln),$ and this vanishes on any element
      that contains $1,$ so $\obs{\nuln}$ is in $\hhcn [n+1] {\o A}
      A.$ 

Let $\mu^{n+1} \in \hhcn [n+1] {\o A}
      A_{-1}$ be an arbitrary element. Assume first that $n \geq 2.$  By
      \ref{prop:obs-are-cycles-for-algebras}.(1), $\nuln + \mu^{n+1} =
      \mu^{\leq n+1} + \msu$ is an \A {n+1}-algebra structure exactly
      when $\dab
      [AA](\mu^{n+1}) = \obs{\nuln}.$ If $n = 1,$ then  $\obs{\nuln} =
      0,$ and so by
      \ref{prop:obs-are-cycles-for-algebras} again, $\nu^{1} + \mu^{2}
      + \msu$ is an \Ai-algebra if $\dab(\mu^{2} + \msu) = 0.$ The
      other parts are proved analogously.
    \end{proof}

We now formulate the module analogues of the definitions and
results of this section. Proofs are not included, but are similar to
their algebra analogues.

\begin{lem}\label{lem:ainf-mod-iff-an-mod-all-n}
Let $(A, \nu)$ be a nonunital \Ai-algebra.
\begin{enumerate}
\item An element $p_{M} \in \hhcu A {\End M}_{0}$ is an \Ai-module
  structure over $(A, \nu)$ if and only if $p_{M}^{\leq n-1}$ is an \A
  n-module structure over $(A, \nuln [n-1])$ for all $n \geq 1.$
  
\item An element $f \in \hhcu A {\Hom {} M N}_{1}$ is a morphism of
  \Ai-modules $(M, p_{M}) \to (N, p_{N})$ if and only if $f^{\leq n-1}$ is a morphism of \A
  n-modules $(M, p_{M}^{\leq n-1}) \to (N, p_{N}^{\leq n-1})$ for all $n \geq 1.$

  \item An element $f \in \hhcu A {\Hom M
      N}_{2}$ is a homotopy between morphisms of \Ai-modules $f,g: (M, p_{M}) \to (N, p_{N})$ if and only if
    $r^{\leq n-1}$ is a homotopy between $f^{\leq
      n-1}$ and $g^{\leq n -1}$ for all $n \geq 1.$
\end{enumerate}
\end{lem}

\begin{defn}
  Let $(A, \nuln)$ be a nonunital \A n-algebra.
  \begin{enumerate}
  \item If $(M, p^{\leq n-1})$ is an \A n-module over
    $(A, \nuln [n-1]),$ an element
    $p^{n} \in \hhcn A {\End M}_{0}$ \emph{extends
      $p^{\leq n-1}$} if $(M, p^{\leq n})$ is an \A {n+1}-module
     over $(A, \nuln).$ The \emph{obstruction} (to extending
    $p^{\leq n-1}$) is
    \begin{displaymath}
      \obs{p^{\leq n-1}} = -\piln{\nuend^{\leq n} * p^{\leq n-1} - p^{\leq n-1} \circ
        \nuln} \in \hhcn A {\End M}_{-1}.
    \end{displaymath}
\item If $(M, p^{\leq n}_{M})$ and $(N, p^{\leq n}_{M})$ are \A
  {n+1}-modules over $(A, \nuln)$ and $f^{\leq n-1}: (M, p^{\leq
    n-1}_{M})\to (N, p^{\leq n-1}_{M})$ a morphism of \A n-modules, an
  element $f^{n} \in \hhcn A {\Hom {}M N}_{1}$ \emph{extends $f^{\leq
      n-1}$} if $f^{\leq n}$ is a morphism of \A {n+1}-modules $(M,
  p^{\leq n}_{M}) \to (N, p^{\leq n}_{M}).$ The
  \emph{obstruction} (to extending $f^{\leq n-1}$) is
\begin{gather*}
\obs{f^{\leq n-1}} = - \piln{p_{N}^{\leq n} \star f^{\leq n - 1} +
  f^{1 \leq i \leq n-1}\circ \nuln + f^{\leq n-1} \star p_{M}^{\leq
                     n}}\\
  \in \hhcn A {\Hom {} M N}_{0}.
\end{gather*}
  
\item If $f^{\leq n}, g^{\leq n}: (M,
  p^{\leq n}_{M}) \to (N, p^{\leq n}_{M})$ are morphisms of \A
  {n+1}-modules, and $r^{\leq n-1}$ is a homotopy between $f^{\leq n-1}$
  and $g^{\leq n-1},$ an element $r^{n} \in \hhcn A {\Hom {} M N}_{2}$
  \emph{extends $r^{\leq n - 1}$} if $r^{\leq n}$ is a homotopy between
  $f^{\leq n}$ and $g^{\leq n}.$ The \emph{obstruction} (to extending
  $r^{\leq n-1}$) is
\begin{gather*}
\obs{r^{\leq n - 1}} = f^{n} - g^{n} - p_{N}^{\leq n} \star r^{\leq n
  -1} - r^{\leq n-1} \circ \nuln - r^{\leq n - 1}\star p_{M}^{\leq
                       n}\\
  \in \hhcn A {\Hom {} M N}_{1}.
\end{gather*}
  \end{enumerate}
\end{defn}

\begin{prop}\label{prop:obstructions-for-modules-are-cycles}
  Let $(A, \nuln)$ be a nonunital \A n-algebra.
  \begin{enumerate}
  \item If $(M, p^{\leq n-1})$ is an \A n-module over
    $(A, \nuln [n-1]),$ then the obstruction $\obs{p^{\leq n -1}}$ is a
    cycle in $(\hhcn A {\End M}, \dab [A\End]).$ An
    element $p^{n}$ in $\hhcn A {\End M}_{0}$ extends $p^{\leq n-1}$
    if and only if $\dab [A\End](p^{n}) = \obs{p^{\leq n-1}}.$
    
\item If $(M, p^{\leq n}_{M})$ and $(N, p^{\leq n}_{M})$ are \A
  {n+1}-modules over $(A, \nuln)$ and $f^{\leq n-1}: (M, p^{\leq
    n-1}_{M})\to (N, p^{\leq n-1}_{M})$ is a morphism of \A n-modules,
  then the obstruction $\obs{f^{\leq n-1}}$ is a cycle in $(\hhcn A {\Hom {} M N}, \dab [A\oHom]).$ An
    element $f^{n}$ in $\hhcn A {\Hom {} MN}_{1}$ extends $f^{\leq n-1}$
    if and only if $\dab [A\oHom](f^{n}) = \obs{f^{\leq n-1}}.$
  
\item If $f^{\leq n}, g^{\leq n}: (M,
  p^{\leq n}_{M}) \to (N, p^{\leq n}_{M})$ are morphisms of \A
  {n+1}-modules, and $r^{\leq n-1}$ is a homotopy between $f^{\leq n-1}$
  and $g^{\leq n-1},$ then the obstruction $\obs{r^{\leq n-1}}$ is a cycle in $(\hhcn A {\Hom {} M N}, \dab [A\oHom]).$ An
    element $r^{n}$ in $\hhcn A {\Hom {} MN}_{2}$ extends $r^{\leq n-1}$
    if and only if $\dab [A\oHom](r^{n}) = \obs{r^{\leq n-1}}.$
  \end{enumerate}
\end{prop}

\begin{lem}\label{lem:strict-unit-adjustment-for-module}
  Let $(A, 1, \nuln)$ be an \A n-algebra with split unit.
  \begin{enumerate}
  \item If $(M, p^{\leq n-1} = \o p^{\leq n-1} + \gsu)$ is a strictly unital \A n-module over
    $(A, \nuln [n-1]),$ with $\o p^{\leq n-1} \in \hhclnu {\o A} {\End
    M},$ then the obstruction $\obs{p^{\leq n-1}}$ is in
    $\hhcn {\o A} {\End M}_{0}.$ For $\o p^{n} \in \hhcn {\o A}
    {\End M}_{0},$ $\o p^{\leq n} + \gsu$ is a strictly unital \A
    {n+1}-module over $(A, 1, \nuln)$ if $\dab[\overbar{A}\End](\o p^{n})
    = \obs{p^{\leq n-1}}.$
    
\item If $(M, p^{\leq n}_{M})$ and $(N, p^{\leq n}_{M})$ are strictly
  unital \A
  {n+1}-modules over $(A, \nuln),$ and $f^{\leq n-1}: (M, p^{\leq
    n-1}_{M})\to (N, p^{\leq n-1}_{M})$ is a morphism of strictly unital \A
  n-modules, then the obstruction $\obs{f^{\leq n-1}}$ is in $\hhcn {\o A} {\Hom {} M N}.$ For $f^{n} \in \hhcn {\o A} {\Hom {} MN}_{1},$ $f^{\leq n-1}$
    is a morphism $(M, p^{\leq n}_{M}) \to (N, p^{\leq n}_{N})$ if $\dab [\overbar{A}\oHom](f^{n}) = \obs{f^{\leq n-1}}.$

\item If $f^{\leq n}, g^{\leq n}: (M,
  p^{\leq n}_{M}) \to (N, p^{\leq n}_{M})$ are morphisms of strictly
  unital \A
  {n+1}-modules and $r^{\leq n-1}$ is a strictly unital homotopy between $f^{\leq n-1}$
  and $g^{\leq n-1},$ then the obstruction $\obs{r^{\leq n-1}}$ is in
 $\hhcn {\o A} {\Hom {} M N}.$ For
     $r^{n} \in \hhcn {\o A} {\Hom {} MN}_{2},$  $r^{\leq n}$
    is a homotopy between $f^{\leq n}$ and $g^{\leq n}$ if $\dab
    [{\overbar A}\oHom](r^{n}) = \obs{r^{\leq n-1}}.$
  \end{enumerate}
\end{lem}

\section{Proof of main results}
\begin{proof}[Proof of Theorem \ref{thm:main-transfer-result}]
We first prove part 1. Assume that $n \geq 1,$ $(A, 1, \nuln_{A})$ is an \A
n-algebra with split unit, and the following diagram is
commutative for all $1 \leq i \leq n:$
\begin{equation}\label{eq:comm-diag-in-main-proof}
  \begin{tikzcd}
    (\Pi A)^{\otimes i} \ar[r, "\nu^{i}_{A}"] \ar[d, "q^{\otimes
      i}"'] & \Pi A \ar[d, "q"]\\
    (\Pi B)^{\otimes i} \ar[r, "\nu^{i}_{B}"'] & \Pi B.
  \end{tikzcd}
\end{equation}
This holds for $n = 1$ by hypothesis. We will construct $\nu^{n+1}_{A}$ such that \eqref{eq:comm-diag-in-main-proof} is commutative for $i
= n+1.$ The cases $n = 1$ and $n \geq 2$ require different proofs, but
both use the following morphisms of chain complexes:
\begin{align*}
\varphi &= \hhcn [n+1] {\o A} q: (\hhcn [n+1] {\o A} A, \dab
          [\overbar{A}A]) \to (\hhcn [n+1] {\o A} B, \doab),\\
  \phi &= \hhcn [n+1] q {B}: (\hhcn [n+1] {B} B, \dab [BB])
 \to (\hhcn [n+1] {A} B, \dab).
\end{align*}
Since $((\Pi \o A)^{\otimes n+1}, \delta_{\otimes})$ is a
semiprojective complex and $q$ is a surjective
quasi-isomorphism,\footnote{To see that $(\Pi \o A)^{\otimes n+1},
  \delta_{\otimes})$ is semiprojective, note that $\o A = \ker( (A,
  d_{A}) \to (k, 0))$ is semiprojective, and the tensor product of two
semiprojectives is semiprojective.} $\varphi$ is also a surjective
quasi-isomorphism. These maps fit into
the following diagram, where the unlabeled morphism is inclusion,
\begin{displaymath}
  \begin{tikzcd}
    & \hhcn [n+1] {\o A} A \ar[d, "\simeq"' pos = .3, "\varphi" pos = .3, twoheadrightarrow]\\
    & \hhcn [n+1] {\o A} B \ar[d]\\
    \hhcn [n+1] B B \ar[r, "\phi"'] & \hhcn [n+1] {A} B.
  \end{tikzcd}
\end{displaymath}

Assume $n = 1.$ We construct $\nu^{2}_{A}$ such that
\eqref{eq:comm-diag-in-main-proof} is commutative. Consider the
element $\zeta = \phi(\nu^{2}_{B}) - q\msu^{2} \in \hhcn [2] A
B_{-1}.$ We have $\zeta[a|1] = 0 = \zeta[1|a]$ for all $a
\in A,$ since $\nu_{B}$ is strictly unital and $q[1] = [1].$ Thus
$\zeta$ is in $\hhcn [2] {\o A} B_{-1}.$ Using the surjectivity of
$\varphi,$ choose $\mu^{2} \in \hhcn [2] {\o A} A$ such that
$\varphi(\mu^{2}) = \zeta.$ We have, using that $\varphi$ is a
morphism of chain complexes for the first equality,
\begin{align*}
\varphi(\dab [\overbar{A}A](\mu^{2})) = \doab(\varphi(\mu^{2})) &=
\doab(\zeta)\\ 
&= \dab(\phi(\nu^{2}_{B})) - q\dab [AA](\msu^{2})\\
&= \phi(\dab [BB](\nu^{2}_{B})) - \varphi(\dab[AA](\msu^{2})),
\end{align*}
where we can write $q\dab [AA](\msu^{2}) = \varphi(\dab[AA](\msu^{2}))$
since $\dab[AA](\msu^{2})$ is in $\hhcn [2] {\o A} A.$ Since
$\nu^{\leq 2}_{B}$ is an \A 2-algebra structure, and
$\obs{\nu^{1}_{B}} = 0,$ we have $\dab [BB](\nu^{2}_{B}) = 0$ by
\ref{prop:obs-are-cycles-for-algebras}.(1). Rearranging
the equation directly above shows $\varphi( \dab [\overbar{A}A](\mu^{2}) + \dab
[AA](\msu^{2})) = 0,$ i.e., $ \dab [\overbar{A}A](\mu^{2}) + \dab
[AA](\msu^{2}) \in \ker \varphi.$ Since $\varphi$ is a surjective
quasi-isomorphism, $\ker \varphi$ is acyclic, and thus there exists
$\t \mu^{2}$ with $\dab [\overbar{A}A](\t \mu^{2}) =  \dab [\overbar{A}A](\mu^{2}) + \dab
[AA](\msu^{2}).$ Set $\mu^{2}_{A} = \mu^{2} - \t \mu^{2}.$ We then
have $\dab [\overbar{A}A](\mu^{2}_{A}) + \dab [AA](\msu^{2}) = 0,$ and
so
by \ref{lem:sual-obs-live-in-subcomplex}.(1), setting $\nu_{A}^{2} =
\mu^{2}_{A} + \msu^{2}$ makes $(A, 1, \nu^{\leq 2}_{A})$ into an \A
2-algebra with split unit. Moreover,
\begin{align*}
  \phi(\nu^{2}_{B}) &= q(\mu^{2} + \msu^{2})\\
  &= q(\mu^{2} - \t\mu^{2} + \msu^{2}) = qv^{2}_{A},
\end{align*}
where the second equality uses that $\t \mu^{2}$ is in $\ker \varphi.$
Thus \eqref{eq:comm-diag-in-main-proof} is commutative for
$i = 2.$

Assume now that $n \geq 2.$ We continue to use the morphisms
$\varphi$ and $\phi$ defined above. We have, where the first equality is by definition,
\begin{align*}
  \phi(\obs{\nuln_{B}}) &= \piln [n+1]{\nuln_{B} \circ \nuln_{B}}
                          q^{\otimes n+1}\\
  &= q \piln[n+1]{\nuln_{A} \circ \nuln_{A}} = \varphi(\obs{\nuln_{A}}).
\end{align*}
(The second equality follows from the
commutativity of \eqref{eq:comm-diag-in-main-proof} for all $1 \leq i
\leq n,$ and the third from the fact that $\obs{\nuln_{A}}$ is in
$\hhcn [n+1] {\o A} A,$ by
\ref{lem:sual-obs-live-in-subcomplex}.(1).) Since $\nu_{B}$ is strictly
unital and $q[1] = [1],$ the element $\phi(\nu_{B}^{n+1})$ is in
$\hhcn [n+1] {\o A}B.$ Using the surjectivity of $\varphi,$ choose
$\mu^{n+1}$ in $\hhcn [n+1] {\o A} A$ with $\varphi(\mu^{n+1}) =
\phi(\nu_{B}^{n+1}).$ We then have, where the first equality follows
from \ref{prop:obs-are-cycles-for-algebras}.(1),
\begin{align*}
 0 = \phi( \dab [BB](\nu^{n+1}_{B}) - \obs{\nuln_{B}}) &=
  \doab(\phi(\nu^{n+1}_{B})) - \varphi(\obs{\nuln_{A}})\\
&= \doab(\varphi(\mu^{n+1}))- \varphi(\obs{\nuln_{A}})\\
&= \varphi(\dab [\overbar{A}A](\mu^{n+1}) - \obs{\nuln_{A}}),
\end{align*}
i.e., $\dab [\overbar{A}A](\mu^{n+1}) - \obs{\nuln_{A}} \in \ker
\varphi.$ Since this element is a cycle, by
\ref{lem:sual-obs-live-in-subcomplex}.(1), and $\ker \varphi$ is
acyclic, there exists $\t \mu^{n+1}$ in $\hhcn [n+1] {\o A} A$ with $\dab [\overbar{A}A](\t \mu^{n+1}) = \dab
[\overbar{A}A](\mu^{n+1}) - \obs{\nuln_{A}}.$ Set $\mu^{n+1}_{A} =
\mu^{n+1} - \t \mu^{n+1},$ so $\dab [\overbar{A}A](\mu^{n+1}_{A}) =
\obs{\nuln_{A}}.$ By \ref{lem:sual-obs-live-in-subcomplex}.(1), $\nuln
[n+1]_{A} = \muln[n+1]_{A} + \msu$ is a strictly unital \A
{n+1}-algebra structure. Moreover, since $\varphi(\mu^{n+1}_{A}) =
\varphi(\mu^{n+1}) = \phi(\nu^{n+1}_{B}),$ and $\nuln[n+1]_{A}, \nuln[n+1]_{B}$ are both strictly unital, \eqref{eq:comm-diag-in-main-proof} is commutative for $i = n+1.$ It now follows by induction and Lemma
\ref{lem:ainf-iff-analln} that there is a strictly unital element $\nu_{A} \in \hhc
{A} A_{-1}$ such that $(A, 1, \nu_{A})$ is an \Ai-algebra
with split unit and such that \eqref{eq:comm-diag-in-main-proof} is
commutative for all $i \geq 1.$ Thus $q$ is a strict morphism
$(A,\nu_{A}) \to (B, \nu_{B}).$

If $(B, \nu_{B})$ is augmented, then one can replace $\hhcn [n+1] B B$ with
$\hhcn [n+1] {\o B} {\o B},$ $\hhcn [n+1] A B$ with $\hhcn [n+1] {\o
  A} {\o B},$ and $\hhcn [n+1] {\o A} A$ with $\hhcn [n+1] {\o A} {\o
  A},$ and mimick the previous proof to construct $\mu_{A} \in
\hhc {\o A} {\o A}$ such that $\nu_{A} = \mu_{A} + \msu$ is an
augmented \Ai-structure and $q$ is a strict morphism.

We now prove part 2. Write $\alpha = 
\beta + \gsu,$ with $\beta \in \hhc {\o C} B_{0}.$ Assume that $n \geq
1,$ and $\gamma^{\sleq n} + \gsu$ is a strictly
unital morphism of \A n-algebras such that $q \gamma^{\sleq n} = 
\beta^{\sleq n}.$ This holds for $n = 1$ by hypothesis.
Define $\psi$ as follows, 
\begin{displaymath}
\psi =  \hhcn [n+1] {\o C} q: (\hhcn [n+1] {\o C} A, \dab [\overbar{C}A]) \to (\hhcn [n+1] {\o C} B,
\dab [\overbar{C}B]).
\end{displaymath}
Since $q$ is a surjective quasi-isomorphism and $(\Pi \o C)^{\otimes
  n+1}$ is semiprojective, $\psi$ is a surjective quasi-isomorphism. We now calculate:
\begin{align*}
  \psi(\obs{\delta^{\sleq n}}) &= q \piln [n+1] { \gamma^{\sleq n} \circ \nuln [n+1]_{C} -
                     \nuln[n+1]_{A} * \gamma^{\sleq n}}\\
&= \piln[n+1] {(q\gamma^{\sleq n}) \circ \nuln[n+1]_{C} - (q \nuln[n+1]_{A})*\gamma^{\sleq n}
  }\nonumber\\
&= \piln[n+1]{\bln \circ \nuln[n+1]_{C} - (\nuln[n+1]_{B}*q)*\gamma^{\sleq n}}
  \nonumber\\
&= \piln[n+1]{\bln \circ \nuln[n+1]_{C} - \nuln[n+1]_{B}*\bln} = \obs{\aln}.
  \nonumber
\end{align*}
Using the surjectivity of $\psi,$ choose $\t \gamma \in \hhcn [n+1]{\o
  C} A_{0}$ with $\psi(\t \gamma) = \beta^{n+1}.$ Thus
\begin{align*}
\psi(\dab[\overbar{C}A](\t \gamma) - \obs{\delta^{\sleq n}}) &= \dab[\overbar{C}B] \psi(\t \gamma) -
  \obs{\aln}\\
&= \dab[\overbar{C}B](\beta^{n+1}) - \obs{\aln} = 0,\nonumber
  \end{align*}
where the last equality holds by
\ref{lem:sual-obs-live-in-subcomplex}.(2). Thus
$\dab[\overbar{C}A](\t \gamma) - \obs{\delta^{\sleq n}}$ is in
$\ker \psi.$ By \ref{prop:obs-are-cycles-for-algebras}.(2)
$\obs{\delta^{\sleq n}}$ is a cycle, and thus $\dab[CA](\t \gamma) -
\obs{\delta^{\sleq n}}$ is also a cycle. Since $\psi$ is a quasi-isomorphism,
$\ker \psi$ is acyclic, and thus there exists $\epsilon \in \ker
\psi$ with
$\dab[\overbar{C}A](\epsilon) = \dab[\overbar{C}A](\t \gamma) -
\obs{\delta^{\sleq n}}.$ Set
$\gamma^{n+1} = \t \gamma - \epsilon,$ so $\dab[\overbar{C}A](\gamma^{n+1})
- \obs{\delta^{\sleq n}} = 0.$ By \ref{lem:sual-obs-live-in-subcomplex}.(2),
${\delta}^{\sleq n} = {\gamma}^{\sleq n+1} + \gsu$ is a strictly
unital morphism of \A {n+1}-algebras. Also, $p \gamma^{n+1} = \psi(\gamma^{n+1}) = \psi(\t \gamma) =
\beta^{n+1},$ since $\psi(\epsilon) = 0.$ Thus, by induction,
there exists $\gamma \in \hhc {\o C} A_{0},$ such that $\delta = 
\gamma + \gsu$ is a strictly unital morphism with $q \delta = \alpha.$

To prove part 3, let $\delta' = \gamma' + \gsu$ be another strictly
unital morphism lifting $\alpha$ through $q$, and assume that
$H_{0}(\hhcn [n+1] {\o C} A, \dab [\overbar{C}A])= 0$ for all $n \geq 1.$  Since $q \gamma^{1}= q\gamma'^{1},$ $\gamma^{1} - \gamma'^{1}$ is in
  $\ker \psi.$ This is element is a cycle, since $\gamma^{1}$ and $\gamma'^{1}$ are cycles.
 Because $\ker \psi$ is acyclic,
  there exists $r^{1} \in \hhcn [1] {\o C} {A}_{1}$ with
  $\dab[\overbar{C}A](r^{1}) = \gamma^{1} - \gamma'^{1}.$ Assume by
  induction $\rln$ is a
  strictly unital homotopy between $\bln$ and $\gamln.$ Since
  $\obs{\rln}$ is a cycle in $\hhcn [n+1] {\o C} A_{0},$ and
  $H_{0}(\hhcn [n+1] {\o C} A, \dab [CA])= 0$, there exists
  $r^{n+1}$ with $\dab[\overbar{C}A](r^{n+1}) = \obs{\rln}.$ Thus by
  \ref{lem:sual-obs-live-in-subcomplex}.(3), $\rln [n+1]$ is a
  strictly unital homotopy between $\gamma^{\sleq n+1}$ and
  $\gamma'^{\sleq n+1}.$ Now by induction, and
  \ref{lem:ainf-iff-analln}.(3), there exists a strictly unital
  homotopy between $\delta$ and $\delta'.$
\end{proof}

\begin{proof}[Proof of Theorem \ref{thm:transfer-of-mod-strs}]
 We first prove part 1. Write $p_{M} = \o p_{M} + \gsu$ with $\o p_{M}
 \in \hhc {\o A} {\End M}_{0}.$ Assume that $n \geq 1$ and $\o
 p_{G}^{\leq n-1}$ is an element in $\hhcln [n-1] {\o A} {\End G}_{0}$
 such that $(G,\o p_{G}^{\leq n-1} + \gsu)$ is a strictly unital
 module over $(A, 1, \nuln)$ and the following diagram is commutative
 for all $0 \leq i \leq n-1:$
 \begin{equation}\label{eq:comm-diag-in-proof-mod-transfer}
  \begin{tikzcd}
    (\Pi \o A)^{\otimes i} \ar[r, "\o p^{i}_{G}"] \ar[d, "\o
    p^{i}_{M}"'] & \Pi \End G \ar[d, "\Pi q_{*}"]\\
    \Pi \End M \ar[r, "\Pi q^{*}"'] & \Pi \Hom {} G M.
  \end{tikzcd}
\end{equation}
This holds for $n = 1$ by hypothesis (note that $\o p^{0} =
p^{0}$). The morphism $q_{*}: (\End G, \dhom) \to (\Hom {} G
M, \dhom)$ is a surjective quasi-isomorphism, since $q$ is a
surjective quasi-isomorphism and $(G,
p^{0}_{G})$ is semiprojective. Consider the following morphisms of
complexes:
\begin{align*}
  \varphi &= \hhcn {\o A} {q_{*}}: (\hhcn {\o A} {\End G},
            \dab [\overbar{A}\End]) \to  (\hhcn {\o A} {\Hom {} G M},
            \dab[\overbar{A}\oHom])\\
\phi &= \hhcn {\o A} {q^{*}}: (\hhcn {\o A} {\End M},
            \dab [\overbar{A}\End]) \to  (\hhcn {\o A} {\Hom {} G M},
            \dab[\overbar{A}\oHom]).
\end{align*}
Since $q_{*}$ is a surjective quasi-isomorphism and $((\Pi \o
A)^{\otimes n}, \delta_{\otimes})$ is semiprojective, $\varphi$ is a
surjective quasi-isomorphism. We have the following diagram:
\begin{displaymath}
  \begin{tikzcd}
    & \hhcn {\o A} {\End P} \ar[d, "\simeq"' pos = .3, "\varphi" pos = .3, twoheadrightarrow]\\
    \hhcn {\o A} {\End M} \ar[r, "\phi"'] & \hhcn {\o A} {\Hom {} P M}.
  \end{tikzcd}
\end{displaymath}
Using the surjectivity of $\varphi,$ choose $\o p^{n}$ in $\hhcn {\o
  A}{\End P}$ with $\varphi(\o p^{n}) = \phi(\o p_{M}^{n}).$ Using that
$\varphi$ and $\phi$ are morphisms of complexes, we have $\varphi(\dab
[\overbar{A}\End](\o p^{n})) = \phi(\dab
[\overbar{A}\End](\o p^{n}_{M})).$
We also claim that
\begin{equation}
\phi(\obs{p_{M}^{\leq
    n-1}}) = \varphi(\obs{p_{G}^{\leq n-1}}).\label{eq:proof-module-main-thm}
\end{equation}
Assuming the claim (we give a proof in the
next paragraph), we then have
\begin{displaymath}
\varphi(\dab
[\overbar{A}\End](\o p^{n}) - \obs{\o p_{G}^{\leq n-1}}) = \phi(\dab
[\overbar{A}\End](\o p^{n}_{M}) - \obs{\o p_{M}^{\leq n-1}}) = 0,
\end{displaymath}
where the last equality follows from
\ref{lem:strict-unit-adjustment-for-module}.(1). Thus $\dab
[\overbar{A}\End](\o p^{n}) - \obs{\o p_{G}^{\leq n-1}}$ is in $\ker
\varphi.$ Since $\obs{\o p_{G}^{\leq n-1}}$ is a cycle in $(\hhcn {\o
  A} {\End G}, \dab
[\overbar{A}\End])$ and $\ker \varphi$ is acyclic (since $\varphi$ is
a quasi-isomorphism), there exists $\t p^{n} \in \ker \varphi \subseteq
\hhcn {\o A} {\End G}$
with $\dab
[\overbar{A}\End](\t p^{n}) = \dab
[\overbar{A}\End](\o p^{n}) - \obs{\o p_{G}^{\leq n-1}}.$ Set $\o
p^{n}_{G} = \o p^{n} - \t p^{n},$ so $\dab
[\overbar{A}\End](\o p^{n}_{G}) = \obs{\o p_{G}^{\leq n-1}}.$ By
\ref{lem:strict-unit-adjustment-for-module}.(1), $\o p_{G}^{\leq n} +
\gsu$ is a strictly unital \A n-module structure on $G.$ Moreover, 
\begin{displaymath}
\Pi q_{*}(\o p^{n}_{G}) = \varphi(\o p^{n} - \t p^{n}) = \varphi(\o
p^{n}) = \phi(\o p^{n}_{M}) = \Pi q^{*}(\o p^{n}_{M}),
\end{displaymath}
where the first equality is by definition of $\varphi$ and $\o
p^{n}_{G},$ the second because $\t p^{n}$ is in $\ker \varphi,$ the
third by choice of $\o p^{n},$ and the fourth by definition of $\phi.$
Thus \eqref{eq:comm-diag-in-proof-mod-transfer} is commutative for $i
= n.$ It now follows by induction and
\ref{lem:ainf-mod-iff-an-mod-all-n}.(1) that there exists $\o p_{G} \in
\hhc {\o A} {\End G}_{0}$ such that $(G, \o p_{G} + \gsu)$ is a
strictly unital \Ai-module over $(A, 1, \nu).$ By induction
\eqref{eq:comm-diag-in-proof-mod-transfer} is commutative for all $i,$
and thus $q$ is a strict morphism of strictly unital modules.

We now prove that \eqref{eq:proof-module-main-thm}
holds. First, we note:
\begin{align*}
  \phi(\obs{p_{M}^{\leq
  n-1}}) &= -\Pi q^{*}(\nuend^{\leq n} * p_{M}^{\leq n-1} - p_{M}^{\leq n-1}
  \circ \nuln)^{\leq n}\\
&= -\Pi q^{*}(\nuend^{1} p_{M}^{\leq n-1} + \nuend^{2} * p_{M}^{\leq n-1} - p_{M}^{\leq n-1}
  \circ \nuln)^{\leq n}.
  \intertext{and}
    \varphi(\obs{p_{G}^{\leq
  n-1}}) &= -\Pi q_{*}(\nuend^{\leq n} * p_{G}^{\leq n-1} - p_{G}^{\leq n-1}
  \circ \nuln)^{\leq n}\\
&= -\Pi q_{*}(\nuend^{1} p_{G}^{\leq n-1} + \nuend^{2} * p_{G}^{\leq n-1} - p_{G}^{\leq n-1}
  \circ \nuln)^{\leq n}.
\end{align*}
Since $\Pi q^{*}$ is a morphism of complexes and $\nuend^{1} =
\delta_{\Pi \oHom},$ we have $\Pi q^{*} \nuend^{1} = \delta_{\Pi
  \oHom} \Pi q^{*}.$ This gives the first equality below, the
second follows from \eqref{eq:comm-diag-in-proof-mod-transfer}, and
the third is analogous to the first,
\begin{displaymath}
\Pi q^{*} \nuend^{1} p_{M}^{\leq n -1} = \delta_{\Pi
  \oHom} \Pi q^{*} p_{M}^{\leq n -1} = \delta_{\Pi
  \oHom} \Pi q_{*} p_{G}^{\leq n -1} = \Pi q_{*} \nuend^{1} p_{G}^{\leq n -1}.
\end{displaymath}
Thus the first terms of $\phi(\obs{p_{M}^{\leq
  n-1}})$ and $\varphi(\obs{p_{G}^{\leq
  n-1}})$ agree. Using
\eqref{eq:comm-diag-in-proof-mod-transfer}, we have that $\Pi q^{*}
p_{M}^{\leq n-1} \circ \nuln = \Pi q_{*} p_{G}^{\leq n-1} \circ \nuln,$ so the
third terms also agree.  Working on the second term, where we denote by
$\gamma$ various composition maps, we have
\begin{align*}
  \Pi q^{*} \nuend^{2} p_{M}^{\leq n-1} &= -\Pi s \gamma
                                          (s^{-1})^{\otimes 2} *
                                          p^{\leq n-1}_{M}\\
&= - sq^{*} \gamma (s^{-1})^{\otimes 2} * p_{M}^{\leq n-1}\\
&= -s\gamma(1 \otimes q^{*})(s^{-1})^{\otimes 2}*p_{M}^{\leq n-1}\\
&= s\gamma(s^{-1})^{\otimes 2}(1 \otimes \Pi q^{*})*p^{\leq n-1}_{M}.
\end{align*}
Analogously, $\Pi q_{*} \nuend^{2} p_{G}^{\leq n-1} =
s\gamma(s^{-1})^{\otimes 2}(1 \otimes \Pi q_{*})*p^{\leq n-1}_{G}.$
Using \eqref{eq:comm-diag-in-proof-mod-transfer} one checks these last
two terms agree.

We now prove part 2. Since $(N, p_{N}^{0})$ is semiprojective and $q$
is a surjective quasi-isomorphism, there exists a morphism of chain
complexes $\delta^{0}: (N, p_{N}^{0}) \to (G, p_{G}^{0})$ such that $q
\delta^{0} = \alpha^{0}.$ Assume by induction that there exists $\delta^{\leq n-1}
\in \hhclnu {\o A} {\Hom {} N G}_{1},$ a morphism of strictly
unital \A n-modules $(N, p_{N}^{\leq n-1}) \to (G, p_{G}^{\leq n-1})$
such that $q \delta^{\leq n -1} =
\alpha^{\leq n-1}.$ Set
\begin{displaymath}
\psi = \hhcn {\o A} {q_{*}}: (\hhclnu {\o A} {\Hom {} N G}, \dab
[\overbar{A}\oHom]) \to  (\hhclnu {\o A} {\Hom {} N M}, \dab
[\overbar{A}\oHom]).
\end{displaymath}
Since $(N, p^{0}_{N})$ is semiprojective and $q$ is a surjective
quasi-isomorphism, $\Pi q_{*}: (\Pi \Hom {} N G, \delta_{\Pi \oHom})
\to (\Pi \Hom {} N M, \delta_{\Pi \oHom})$ is also a surjective
quasi-isomorphism. Since $((\Pi \o A)^{\otimes n}, \delta_{\otimes})$
is semi-projective, $\psi$ is also a surjective quasi-isomorphism.

Using the surjectivity of $\psi,$ choose $\t \delta \in \hhcn {\o A}
{\Hom {} N G}$ with $\psi(\t \delta) = \alpha^{n}.$ We claim
$\psi(\obs{\delta^{\leq n-1}}) = \obs{\alpha^{\leq n-1}}.$ This
follows from the fact that $q$ is a strict morphism and that $q
\delta^{\leq n-1} = \alpha^{\leq n-1}.$ We now have:
\begin{displaymath}
\psi( \dab[\overbar{A}\oHom](\t \delta) - \obs{\delta^{\leq n-1}}) =
\dab[\overbar{A}\oHom](\alpha^{n}) - \obs{\alpha^{\leq n-1}} = 0,
\end{displaymath}
where the last equality holds by
\ref{lem:strict-unit-adjustment-for-module}.(2). Thus,
$\dab[\overbar{A}\oHom](\t \delta) - \obs{\delta^{\leq n-1}}$ is in
$\ker \psi$. This element is a cycle by
\ref{prop:obstructions-for-modules-are-cycles}.(2). Since $\psi$ is a
quasi-isomorphism, $\ker \psi$ is acyclic, and so there exists $\t{\t
  \delta}$ in $\ker \psi$ with $\dab [\overbar{A}\oHom](\t{\t \delta})
= \dab[\overbar{A}\oHom](\t \delta) - \obs{\delta^{\leq n-1}}.$ Set
$\delta^{n} = \t \delta - \t{\t \delta}.$ Since $\dab
[\overbar{A}\oHom](\delta^{n}) = \obs{\delta^{\leq n-1}},$ it follows
from \ref{lem:strict-unit-adjustment-for-module}.(2) that
$\delta^{\leq n}$ is a morphism of strictly unital \A n-modules. Also
$q \delta^{n} = \psi(\delta^{n}) = \alpha^{n}.$ Thus by induction and
\ref{lem:ainf-mod-iff-an-mod-all-n}.(2), there exists $\delta \in
\hhcu {\o A} {\Hom {} N G}$ that is a morphism of strictly unital
\Ai-modules and $q\delta = \alpha.$

The proof of part 3 is analogous to the proof of part 3 of \ref{thm:main-transfer-result}.
\end{proof}

\def\cprime{$'$}


\begin{thebibliography}{10}

\bibitem{AFH}
Luchezar~L. Avramov, Hans-Bjorn Foxby, and Stephen Halperin.
\newblock {Differential graded homological algebra}.
\newblock Preprint, June 21, 2006.

\bibitem{AvHa86}
\mbox{Luchezar~L.} Avramov and Stephen Halperin.
\newblock Through the looking glass: a dictionary between rational homotopy
  theory and local algebra.
\newblock In {\em Algebra, algebraic topology and their interactions
  ({S}tockholm, 1983)}, volume 1183 of {\em Lecture Notes in Math.}, pages
  1--27. Springer, Berlin, 1986.

\bibitem{1508.03782}
Jesse Burke.
\newblock Higher homotopies and {G}olod rings, 2015.

\bibitem{SUal}
Jesse Burke.
\newblock Strictly unital ${A}_{\infty}$-algebras, 2017.

\bibitem{MR0161898}
Murray Gerstenhaber.
\newblock The cohomology structure of an associative ring.
\newblock {\em Ann. of Math. (2)}, 78:267--288, 1963.

\bibitem{MR3136262}
Phillip Griffiths and John Morgan.
\newblock {\em Rational homotopy theory and differential forms}, volume~16 of
  {\em Progress in Mathematics}.
\newblock Springer, New York, second edition, 2013.

\bibitem{MR662761}
V.~K. A.~M. Gugenheim.
\newblock On a perturbation theory for the homology of the loop-space.
\newblock {\em J. Pure Appl. Algebra}, 25(2):197--205, 1982.

\bibitem{MR1103672}
V.~K. A.~M. Gugenheim, L.~A. Lambe, and J.~D. Stasheff.
\newblock Perturbation theory in differential homological algebra. {II}.
\newblock {\em Illinois J. Math.}, 35(3):357--373, 1991.

\bibitem{MR885535}
V.~K. A.~M. Gugenheim and J.~D. Stasheff.
\newblock On perturbations and {$A_\infty$}-structures.
\newblock {\em Bull. Soc. Math. Belg. S\'er. A}, 38:237--246 (1987), 1986.

\bibitem{MR1650134}
Mark Hovey.
\newblock {\em Model categories}, volume~63 of {\em Mathematical Surveys and
  Monographs}.
\newblock American Mathematical Society, Providence, RI, 1999.

\bibitem{MR1109665}
Johannes Huebschmann and Tornike Kadeishvili.
\newblock Small models for chain algebras.
\newblock {\em Math. Z.}, 207(2):245--280, 1991.

\bibitem{MR580645}
T.~V. Kadei{\v{s}}vili.
\newblock On the theory of homology of fiber spaces.
\newblock {\em Uspekhi Mat. Nauk}, 35(3(213)):183--188, 1980.
\newblock International Topology Conference (Moscow State Univ., Moscow, 1979).

\bibitem{MR1854636}
Bernhard Keller.
\newblock Introduction to {$A$}-infinity algebras and modules.
\newblock {\em Homology Homotopy Appl.}, 3(1):1--35, 2001.

\bibitem{MR2596638}
M.~Kontsevich and Y.~Soibelman.
\newblock Notes on {$A_\infty$}-algebras, {$A_\infty$}-categories and
  non-commutative geometry.
\newblock In {\em Homological mirror symmetry}, volume 757 of {\em Lecture
  Notes in Phys.}, pages 153--219. Springer, Berlin, 2009.

\bibitem{MR1882331}
Maxim Kontsevich and Yan Soibelman.
\newblock Homological mirror symmetry and torus fibrations.
\newblock In {\em Symplectic geometry and mirror symmetry ({S}eoul, 2000)},
  pages 203--263. World Sci. Publ., River Edge, NJ, 2001.

\bibitem{MR1361938}
Igor K{\v{r}}{\'{\i}}{\v{z}} and J.~P. May.
\newblock Operads, algebras, modules and motives.
\newblock {\em Ast\'erisque}, (233):iv+145pp, 1995.

\bibitem{LH}
K.~Lefèvre-Hasegawa.
\newblock {\em Sur les A-infini catégorie}.
\newblock PhD thesis, University of Paris 7.

\bibitem{MR1672242}
S.~A. Merkulov.
\newblock Strong homotopy algebras of a {K}\"ahler manifold.
\newblock {\em Internat. Math. Res. Notices}, (3):153--164, 1999.

\bibitem{MR2844537}
Alain Prout{\'e}.
\newblock {$A_\infty$}-structures. {M}od\`eles minimaux de {B}aues-{L}emaire et
  {K}adeishvili et homologie des fibrations.
\newblock {\em Repr. Theory Appl. Categ.}, (21):1--99, 2011.
\newblock Reprint of the 1986 original.

\bibitem{Sp88}
N.~Spaltenstein.
\newblock Resolutions of unbounded complexes.
\newblock {\em Compositio Math.}, 65(2):121--154, 1988.

\end{thebibliography}

\end{document}